\documentclass[11pt]{amsart}
\usepackage{amssymb, amsmath}
\usepackage{mathtools}
\usepackage{graphicx}
\usepackage{color}
\usepackage[all]{xy}

\usepackage{float}
\usepackage{color}
\usepackage{graphicx}
\usepackage[dvips]{epsfig}

\usepackage{graphics}
\usepackage{amsbsy, amsmath,amsfonts,amssymb,amsthm,amscd,amssymb,latexsym,}
\usepackage{etex}
\usepackage{hyperref}
\newtheorem{theorem}{\bf Theorem}[section]

\newtheorem{lemma}[theorem]{\bf Lemma}
\newtheorem{proposition}[theorem]{\bf Proposition}

\newtheorem{remark}[theorem]{\bf Remark}

\usepackage[dvipsnames]{xcolor}

\usepackage{color}
 	\definecolor{azure(colorwheel)}{rgb}{0.0, 0.5, 1.0}
  	\definecolor{awesome}{rgb}{1.0, 0.13, 0.32}

\begin{document}
\title[Affine curvature lines of surfaces in $3$-space]
{Affine curvature lines of surfaces in $3$-space\\
}
\author[M.~Barajas]{Mart\'{i}n Barajas S.}
\address[M.~Barajas]{Instituto de Matem\'aticas e Estat\'istica,
Universidade Federal de Goi\'as-Goi\^ania-GO, Brazil}
\email{mbarajas@ufg.br}

\author[M. Craizer]{Marcos Craizer}
\address[M. Craizer]{Departamento de Matem\'atica, PUC-Rio, Rio de Janeiro-RJ, Brazil}
\email{craizer@puc-rio.br}
\author[R.~Garcia]{Ronaldo Garcia}
\address[R.~Garcia]{Instituto de Matem\'aticas e Estat\'istica,
Universidade Federal de Goi\'as-Goi\^ania-GO, Brazil}
\email{ragarcia@ufg.br}
%
%
%
\begin{abstract}
In this work we study the affine principal lines of surfaces in $3$-space.
 We consider the binary differential equation  of the affine curvature lines and obtain the topological
 models of these curves  near the affine umbilic points (elliptic and hyperbolic).
 We also describe  the generic  behavior of affine curvature lines in the neighborhood of points with double eigenvalues  (but not umbilics) of the affine shape operator and  parabolic points.
\end{abstract}
\maketitle
\setlength{\baselineskip}{14pt}
%
%
\section{Introduction}
In the context of Euclidean differential geometry of surfaces immersed in $\mathbb{R}^3$, the principal configuration of a surface 
consists of the  two orthogonal principal foliations: the leaves are the principal lines and the umbilic points are the singularities of both foliations. Near the umbilic points the results of Darboux \cite{Da-1896}
say that there are three distinct topological models, see Fig. \ref{fig1}. See also \cite{Bruce1989} and  \cite{GS-1982}.

 It is worth to mention that G. Monge \cite{Mo-1796} 
described the global behavior of principal lines in the ellipsoid $x^2/a^2+y^2/b^2+z^2/c^2=1$. For a  survey about the origins and recent developments in this subject of research   see \cite{Garcia2009} and \cite{GaSo-2016}.

The relation between Euclidean and affine principal lines was considered by Su Buchin, \cite{Buchin}. But in general, even the umbilic points are not coincident. The Euclidean  umbilic surfaces are  planes and spheres. On the order hand, the affine umbilic surfaces, also called affine spheres, include  all non degenerate quadrics, some cubic surfaces and many others, see  \cite[Chapter 3]{Simon},   \cite[Chapter 3]{Nomizu1994}, \cite{Craizer-2008}, \cite{Fox-2012}, \cite{Milan-2017}. For a survey on affine spheres see \cite{Loftin2010}.

In this work we study the principal lines in the context of affine differential geometry of surfaces immersed in $3$-space.
The main goal here is to describe the  local behavior of   affine principal lines near the affine umbilic points (elliptic and hyperbolic), points with double eigenvalues of the affine shape operator (but not umbilics) and the parabolic set.

An affine   configuration is the triple formed by the two orthogonal affine  foliations, whose leaves are the affine   curvature lines and its singular set. 

Two affine  configurations are said to be locally topologically equivalent if there is a germ of homeomorphim sending the affine  lines of curvature of the fisrt configuration into the corresponding ones of the second affine configuration and also preserving the singularities (affine elliptic and hyperbolic umbilic points,   points with double eigenvalues,  parabolic set, discriminant set).

Analogously, two binary differential equations are said to be locally topologically equivalent if there is a germ of homeomorphism sending the corresponding integral curves of the first to the second binary equation and preserving the singularities.

The paper is organized as follows. In Section \ref{sec:prelimi} the basic concepts are introduced and the differential equation of affine principal lines is established.
In Section \ref{section3} the local behavior of affine principal lines near affine umbilic points are described, see Propositions \ref{prop_A1+} and \ref{prop_A1-}. 
 In Section \ref{section4} the local behavior of lines of curvature near points with double eigenvalues of the affine shape operator is analyzed. 
 Section \ref{section5} is addressed to the study of affine principal lines near   parabolic points. The main result is Theorem \ref{Thm_confg_princ_affins}.

%
%

\section{Preliminaries}\label{sec:prelimi}

Let $S$ be a smooth ($C^\infty$) surface   in the 3-dimensional affine space. 
Outside the parabolic set, we endow $S$ with the Berwald-Blaschke metric given by
\begin{equation}\label{Iaff}
\mathcal{G}=\mid K_e\mid^{-\frac 14} II_e,
\end{equation}
where $K_e$ is the Euclidean Gaussian curvature and $II_e$ is the Euclidean second fundamental form of $S$. The Berwald-Blaschke metric is also called the  affine first fundamental   form and will denoted by $I_a$.  

There is a single transversal field $\xi$ defined on $S$ such that $d\xi\subset TS$ and the area form on $S$ defined by $\xi$ coincides 
with the area given by the affine first fundamental form. 
 The vector field $\xi$ is called affine normal field (also called Blaschke normal field) and locally it is uniquely determined up to a direction sign (\cite{Blaschke1923,Buchin1983,Calabi1982,Dillen1993,Nomizu1994}).\\

We can write, for $p\in S$ and $v\in T_pS$, 
\begin{equation}\label{eq:xi}
d\xi(p) v=B(p)v,
\end{equation}
where $B(p)$ is a linear transformation of $T_pS$, called affine shape operator. It is well-known that $B$ is self-adjoint with respect to 
$\mathcal{G}$. Thus, if $\mathcal{G}$ is positive-definite, the eigenvalues of $B$ are real and the eigenvectors are $\mathcal{G}$-orthogonal.
The real eigenvalues of $B$ are the affine principal curvatures and the eigenvectors are the affine principal directions.  
We say that $p$ is (affine) umbilic if $B(p)$ is a multiple of the identity. In the region where the eigenvalues 
of $B$ are real and outside umbilic points, there exists a pair of foliations tangent to the affine principal directions, called affine foliations of affine curvature lines.
When the Euclidean Gauss curvature is negative, the eigenvalues of the affine shape operator are not always real. 
A point is called a double $\xi$-direction point if there is a single double affine principal direction \cite{Davis2008}. At such points, both principal directions coincide. The set of 
double $\xi$-directions bounds the region where the affine principal curvature lines are defined.

Consider now the case that $S$ is parameterized by $X:U\subset\mathbb{R}^2\to S$. Let
\begin{equation*}
L=\left|X_u,X_v,X_{uu}\right|,\quad M=\left|X_u,X_v,X_{uv}\right|,\quad N=\left|X_u,X_v,X_{vv}\right|,
\end{equation*}
where $X_u=\frac{\partial X}{\partial u}$, $X_v=\frac{\partial X}{\partial v}$ and $\left|a,b,c\right|$ denotes the determinant of the vectors $a,b,c$. One can verify that the Berwald-Blaschke metric is given by
\begin{equation}\label{eq:metricaG}
\aligned
\mathcal{G}=&g_{11}du^2+2g_{12}dudv+g_{22}dv^2,\\
g_{11}=&\frac{L}{\left|LN-M^2\right|^{\frac{1}{4}}},\quad g_{12}=\frac{M}{\left|LN-M^2\right|^{\frac{1}{4}}},\quad g_{22}=\frac{N}{\left|LN-M^2\right|^{\frac{1}{4}}}.
\endaligned 
\end{equation}

The conormal vector $\nu$ to $S$ at $p$ is defined by $\nu=\left|K_e\right|^{-\frac 14}N_e$ and the affine normal vector $\xi$ is totally determined by the relations
\begin{equation}\label{Conormal_normal}
\left\langle \nu,\xi\right\rangle=1\quad\text{and}\quad\left\langle \nu,\xi_u\right\rangle=\left\langle \nu,\xi_v\right\rangle=0,
\end{equation}
where $\left\langle \_,\_\right\rangle$ denotes the Euclidean inner product (see \cite{Calabi1982}).
One can verify that $\nu$ is given by
\begin{equation}\label{Conormal}
\nu(p)=\frac{1}{\left|LN-M^2\right|^{\frac{1}{4}}} \left(  X_u\times X_v \right).
\end{equation}

From equation \eqref{Conormal_normal} and \eqref{Conormal}  it follows that

 \begin{equation}\label{xiafim} \xi= \frac 1{\left|LN-M^2\right|^{\frac{1}{4}}} (\nu_u \times\nu_v). \end{equation}
 
 For  further reference, see \cite{Buchin1983},  we have in a local chart $(u,v)$ that:
 
 \begin{equation}\label{eq:xicarta}
 \xi=\frac 12 \frac{|LN-M^2|^{\frac 14}}{\sqrt{|LN-M^2|}} \left\{ \frac{ \partial}{\partial u} \left(\frac{ NX_u-MX_v }{\sqrt{|LN-M^2|}} \right) + \frac{ \partial}{\partial v}\left(\frac{ L X_v-M X_u}{\sqrt{|LN-M^2|}} \right)
 \right\}
 \end{equation}
The conormal and the affine normal vectors satisfy the equations
\begin{equation*}
\left|X_u,X_v,\xi\right|=\left|\nu,\nu_u,\nu_v\right|=\left|LN-M^2\right|^{\frac{1}{4}}.
\end{equation*}
The affine structural equations are given by
\begin{eqnarray*}
X_{uu} & = & A^{1}_{11}X_u+A^{2}_{11}X_v+g_{11}\xi, \\
X_{uv} & = & X_{vu}=A^{1}_{12}X_u+A^{2}_{12}X_v+g_{12}\xi, \\
X_{vv} & = & A^{1}_{22}X_u+A^{2}_{22}X_v+g_{22}\xi,
\end{eqnarray*}
where the $A^{i}_{jk}$ are the affine Christoffel symbols.\\

Denote by $(b_{ij})$ the matrix of the affine shape operator in the basis $\{X_u,X_v\}$. We can write
\begin{equation}\label{Der_norm_afin}
\left[
\begin{array}{c}
\xi_u \\
\xi_v 
\end{array}
\right]=
\left[
\begin{array}{cc}
b_{11} & b_{21} \\
b_{12} & b_{22}
\end{array}
\right]
\left[
\begin{array}{c}
X_u \\
X_v 
\end{array}
\right],
\end{equation}
with
\begin{eqnarray*}
b_{11} & = & \left|LN-M^2\right|^{-\frac{1}{4}}\left|\xi_u,X_v,\xi\right|, \\
b_{21} & = & \left|LN-M^2\right|^{-\frac{1}{4}}\left|X_u,\xi_u,\xi\right|, \\
b_{12} & = & \left|LN-M^2\right|^{-\frac{1}{4}}\left|\xi_v,X_v,\xi\right|, \\
b_{22} & = & \left|LN-M^2\right|^{-\frac{1}{4}}\left|X_u,\xi_v,\xi\right|.
\end{eqnarray*}
The affine third fundamental form ($III_a$) is a symmetric bilinear quadratic form, given by
\begin{equation}\label{eq:xico}
\left\langle d\xi,d\nu\right\rangle=ldu^2+2mdudv+ndv^2, where
\end{equation}
\begin{eqnarray*}
l =\left\langle \nu_u,\xi_u\right\rangle=-\left\langle \nu,\xi_{uu}\right\rangle, \\
m =\left\langle \nu_u,\xi_v\right\rangle=\left\langle \nu_v,\xi_u\right\rangle=-\left\langle \nu,\xi_{uv}\right\rangle, \\
n =\left\langle \nu_v,\xi_v\right\rangle=-\left\langle \nu,\xi_{vv}\right\rangle. 
\end{eqnarray*}
In terms of the parametrization $X$, and using \eqref{Der_norm_afin} we obtain
\begin{eqnarray*}
-l & = & b_{11}g_{11}+b_{21}g_{12}, \\
-m & = & b_{11}g_{12}+b_{21}g_{22}, \\
-m & = & b_{12}g_{11}+b_{22}g_{22}, \\
-n & = & b_{12}g_{12}+b_{22}g_{22}.
\end{eqnarray*}
Thus we obtain the   coefficients  $b_{ij}$ in terms of the  affine first and third fundamental forms as follows.
\begin{eqnarray}\label{eq_bijI_III}
\left[
\begin{array}{cc}
b_{11} & b_{21} \\
b_{12} & b_{22}
\end{array}
\right] & = & -\left|g_{11}g_{22}-g_{12}^2\right|^{-1}
\left[
\begin{array}{cc}
l & m \\
m & n
\end{array}
\right]
\left[
\begin{array}{cc}
g_{22}  & -g_{12} \\
-g_{12} & g_{11}
\end{array}
\right] \\
 & = & -\left|LN-M^2\right|^{-\frac{3}{4}}
\left[
\begin{array}{cc}
l & m \\
m & n
\end{array}
\right]
\left[
\begin{array}{cc}
N  & -M \\
-M & L
\end{array}
\right]. \nonumber
\end{eqnarray}

We give now the equation of affine curvature lines.
\begin{proposition}
Let $\gamma(t)=X(u(t),v(t))$ a smooth curve on $S$. Then, $\gamma$ is an affine curvature line if, and only if, $\gamma$ satisfies the binary differential equation
\begin{equation}\label{eq_curv_lines1}
\left(lg_{12}-mg_{11}\right)du^2+\left(lg_{22}-ng_{11}\right)dudv+\left(mg_{22}-ng_{12}\right)dv^2=0
\end{equation}
\end{proposition}
\begin{proof} The affine principal directions $v$ are defined by the eigenvectors equation $d\xi(p)v+\lambda(p)v=0$. These directions are obtained taking the jacobian of the  affine first fundamental forms $I_a=g_{11}du^2+2g_{12}dudv+g_{22}dv^2 $ and the affine third fundamental form $III_a=ldu^2+2mdudv+ndv^2 $ with respect to $(du,dv).$ \qed
\end{proof}
\begin{remark}\label{rem_lin-cur}
As the coefficients of the affine first fundamental form  $I_a$ are proportional to $K_e^{-\frac 14}$  we can regularize this differential equation at the points where $K_e=0$ (parabolic points) and  the affine curvature lines are the integral curves of the   regularized equation
\begin{equation}\label{eq_curv_lines2}
\left(lM-mL\right)du^2+\left(lN-nL\right)dudv+\left(mN-nM\right)dv^2=0.
\end{equation}
\end{remark}

%
\section{Affine curvature lines near  affine umbilic points}\label{section3}
In this section we analyze  the local behavior of affine curvature lines in a neighborhood of  the affine umbilic points, elliptic and hyperbolic.\\
To establish the main result of this section, we will use a special and suitable local parametrization of surface $S$ in a small neighborhood of 
$p$. 
%
\begin{proposition}[Pick normal forms, \cite{Buchin1983,Davis2008}]\label{Pick_norm_forms}
Let $S$ be a smooth surface locally parametrized  by $X(u,v)=(u,v,h(u,v))$ and let $p=X(0,0)$. 
%
Assume that $p=X(0,0)$ is not a parabolic point. Then, by affine change of coordinates,
\begin{enumerate}
	\item[(i)] If $p$ is an elliptic point, $h(u,v)$ can be written as
{\small 
$$
\frac{1}{2}\left(u^2+v^2\right)+\frac{\sigma}{6}\left(u^3-3uv^2\right)+\frac{1}{24}\left({q}_{40}u^4+4{q}_{31}u^3v+
       + 6{q}_{22}u^2v^2+4{q}_{13}uv^3+{q}_{04}v^4\right)
 $$
 $$
       +\frac{1}{120}\left({q}_{50}u^5+5{q}_{41}u^4v+  10{q}_{32}u^3v^2+10{q}_{23}u^2v^3+5{q}_{14}uv^4+{q}_{05}v^5\right)+O(6).
$$
 \item[(ii)] If $p$ is a hyperbolic point, $h(u,v)$ can be written as
$$
\frac{1}{2}\left(u^2-v^2\right)+\frac{\sigma}{6}\left(u^3+3uv^2\right)+\frac{1}{24}\left({q}_{40}u^4+4{q}_{31}u^3v+
       +  6{q}_{22}u^2v^2+4{q}_{13}uv^3+{q}_{04}v^4\right)
 $$
 $$
 +\frac{1}{120}\left({q}_{50}u^5+5{q}_{41}u^4v  
			 +  10{q}_{32}u^3v^2+10{q}_{23}u^2v^3+5{q}_{14}uv^4+{q}_{05}v^5\right)+O(6).
$$
}
\end{enumerate}
%
%
\end{proposition}
\medskip

\begin{lemma}\label{lem:jetnormal}
	Let $X$ be given by Proposition \ref{Pick_norm_forms}. In the elliptic case, the affine normal $\xi$ is given by:
	\begin{equation}\label{eq:jetxie}\aligned
	\xi(u,v)= & \left(\left(\frac 12 \,{\sigma}^{2}-\frac 14 q_{22}-\frac 14 q_{40}\right) u-\frac 14\left(q_{13}+q_{31}\right) v+O(2),\right. \\ &  \left.  
 -\frac 14\left(q_{13}+q_{31}\right)u+\left(\frac 12 \,{\sigma}^{2}-\frac 14 q_{04}- \frac 14 q_{22}\right)v+O(2), 1+O(2)	\right).
	\endaligned
	\end{equation}
	In the hyperbolic case, the affine normal is given by:
	\begin{equation}\label{eq:jetxih}
	\aligned
	\xi(u,v)= & \left(\left(\frac 12 \,{\sigma}^{2}+\frac 14 q_{22}-\frac 14 q_{40}\right)u+\frac 14(q_{13}-q_{31})v+O(2),\right. \\ &  \left.  
-\frac 14(q_{13}-q_{31})u+\left(\frac 12 \,{\sigma}^{2}-\frac 14 q_{04}+\frac 14 q_{22}\right)v+O(2), 1+O(2)\right).
\endaligned
	\end{equation}
Here $O(2)$ means functions of higher order.
\end{lemma}
\begin{proof} 
Follows from equation \eqref{xiafim} and  straightforward calculations.	\qed
\end{proof}
%
%

\begin{lemma}\label{lema:shape}
Let $X$ be given by Proposition \ref{Pick_norm_forms}. In the elliptic case, the affine shape operator $B$ is given by:
\begin{eqnarray}
b_{11} & = & \frac 12 \sigma^2-\frac 14 q_{22}-\frac 14 q_{40}+\left(-\sigma^3-\frac 14 \sigma q_{22}+\frac 54 q_{40}\sigma- \frac 14 q_{32}-\frac 14 q_{50}\right)u+ \nonumber\\
       & + & \left(-\frac 14 q_{41}-\frac 12 \sigma q_{13}-\frac 14 q_{23}\right)v+O(2), \nonumber\\
b_{12} & = & -\frac 14 (q_{13}+q_{31})+\left(-\frac 14 q_{41}-\frac 12 \sigma q_{13}-\frac 14 q_{23}\right)u+\nonumber\\
       & + & \left(\sigma^3-\frac 54 \sigma q_{22}-\frac 14 q_{32}-\frac 34 \sigma q_{04}-\frac 14 q_{14}\right)v+O(2), \nonumber\\
b_{21} & = & -\frac 14 (q_{13}+q_{31})+\left(-\frac 12 q_{31}\sigma-\frac 14 q_{41}-\sigma q_{13}-\frac 14 q_{23}\right)u+ \nonumber\\
       & + & \left(\sigma^3-\frac 14 q_{40}\sigma-\frac 54 \sigma q_{22}-\frac 14 q_{32}-\frac 12 \sigma q_{04}-\frac 14 q_{14}\right)v+O(2), \nonumber\\
b_{22} & = & \frac 12 \sigma^2-\frac 14 q_{22}-\frac 14 q_{04}+\left(\sigma^3-\frac 14 q_{40}\sigma-\frac 54 \sigma q_{22}-\frac 14 q_{32}-\frac 12 \sigma q_{04}-\frac 14 q_{14}\right)u+ \nonumber\\
       &+& \left(-\frac 12 q_{31}\sigma-2\sigma q_{13}-\frac 14 q_{23}-\frac 14 q_{05}\right)v+O(2). \nonumber
\end{eqnarray}
In the hyperbolic case, the affine shape operator $B$ is given by:
\begin{eqnarray}
b_{11} & = & -\frac 12 \sigma^2-\frac 14 q_{22}+\frac 14 q_{40}-\left(-\sigma^3+\frac 14 \sigma q_{22}+\frac 54 q_{40}\sigma+\frac 14 q_{32}-\frac 14 q_{50}\right)u- \nonumber\\
       & - & \left(\frac 12 \sigma q_{13}+\frac 14 q_{23}-\frac 14 q_{41}\right)v+O(2),\nonumber\\
b_{12} & = & \frac 14 (q_{31}- q_{13})-\left(-\frac 14 q_{41}+\frac 12 \sigma q_{13}+\frac 14 q_{23}\right)u- \nonumber\\
       & - & \left(-\sigma^3-\frac 54 \sigma q_{22}-\frac 14 q_{32}+\frac 34 \sigma q_{04}+\frac 14 q_{14}\right)v+O(2), \nonumber\\
b_{21} & = & -\frac 14 (q_{31}- q_{13})-\left(\frac 12 q_{31}\sigma+\frac 14 q_{41}-\sigma q_{13}-\frac 14 q_{23}\right)u- \nonumber\\
       & - &  \left(\sigma^3-\frac 14 q_{40}\sigma+\frac 54 \sigma q_{22}+\frac 14 q_{32}-\frac 12 \sigma q_{04}-\frac 14 q_{14}\right)v+O(2), \nonumber\\
b_{22} & = & -\frac 12 \sigma^2-\frac 14 q_{22}+\frac 14 q_{04}-\left(\sigma^3-\frac 14 q_{40}\sigma+\frac 54 \sigma q_{22}+\frac 14 q_{32}-\frac 12 \sigma q_{04}-\frac 14 q_{14}\right)u- \nonumber\\
       & - & \left(-\frac 12 q_{31}\sigma+2\sigma q_{13}+\frac 14 q_{23}-\frac 14 q_{05}\right)v+O(2).\nonumber
\end{eqnarray}
\end{lemma}
\begin{proof} 
Follows from equation \eqref{eq_bijI_III} in the parametrization of $X$ given by equation \eqref{Pick_norm_forms}.
\qed	
\end{proof}
\medskip
\begin{lemma}\label{lem:edoacl}
Let $X$ be given by Proposition \ref{Pick_norm_forms}. In the elliptic case, $p=X(0,0)$ is an affine umbilic point if and only if $q_{31}=-q_{13}$ and $q_{40}=q_{04}$. In the hyperbolic case, the origin is an affine umbilic point if, and only if, $q_{40}=q_{04}$ and $q_{31}=q_{13}$.
\end{lemma}
\begin{proof}
It follows from Lemma \ref{lema:shape}. \qed
\end{proof}
\medskip
\begin{lemma}\label{lem:difeqACL}
In a neighborhood of an elliptic umbilic point the binary differential equation of the affine curvature is given by:
\begin{equation}
(a_1u+ b_1v)(du^2-dv^2)+2(a_2u+b_2v) dudv+O(2)=0, 
\end{equation}
where
\begin{equation}\label{eq:darboux}\aligned
lM-mL & = & a_1u+ b_1v +O(2)\\
lN-nL & = & 2a_2u+ 
2b_2v+O(2)\\
mN-nM & = & -a_1u- b_1v +O(2)\endaligned
\end{equation}
and 
\begin{eqnarray*}
a_1   & = & \frac{1}{4}\left(2\sigma q_{31}-q_{23}-q_{41}\right) \\
b_1   & = & \sigma^3-\frac{1}{4}\sigma\left(3q_{40}+5q_{22}\right)-\frac{1}{4}\left(q_{14}+q_{32}\right), 
\end{eqnarray*}

\begin{eqnarray*}
a_2  & = &  \sigma^3-\frac{\sigma}{2}\left(2q_{40}+q_{22}\right)+\frac{1}{8}\left(q_{50}-q_{14}\right) \\
b_2  & = &  \frac{\sigma}{2} q_{31}+\frac{1}{8}\left(q_{41}-q_{05}\right). 
\end{eqnarray*}
Here $O(2)$ means higher order terms in relation to the variables $(u,v).$

\end{lemma}
\begin{proof}
Straightforward calculations. \qed
\end{proof}
\medskip
\begin{lemma}
In a neighborhood of a hyperbolic umbilic point the binary differential equation of the affine curvature lines is given by:
\begin{equation}
(a_1u+ b_1v)(du^2+dv^2)+2(a_2u+b_2v) dudv+O(2)=0, 
\end{equation}
where
\begin{equation}\label{eq:darboux}\aligned
lM-mL & = & a_1u+ b_1v +O(2)\\
lN-nL & = & 2a_2u+ 
2b_2v+O(2)\\
mN-nM & = & a_1u +b_1v +O(2)\endaligned
\end{equation}
and 
\begin{eqnarray*}
a_1  & = & \frac{1}{4}\left(2\sigma q_{31}+q_{23}-q_{41}\right)\\
b_1  & = & -\sigma^3+\frac{1}{4}\sigma\left(3q_{40}-5q_{22}\right)+\frac{1}{4}\left(q_{14}-q_{32}\right),
\end{eqnarray*}
\begin{eqnarray*}
a_2   & = & -\sigma^3+\frac{\sigma}{2}\left(2q_{40}-q_{22}\right)-\frac{1}{8}\left(q_{50}-q_{14}\right)\\
b_2   & = & -\frac{\sigma}{2} q_{31}-\frac{1}{8}\left(q_{41}-q_{05}\right). 
\end{eqnarray*}

Here $O(2)$ means higher order terms in relation to the variables $(u,v).$
\end{lemma}

\begin{proof}
Straightforward calculations. \qed
\end{proof}

\medskip

Let $J= {  a_2}\,{  b_1} -{ a_1}\,{  b_2}$ and consider the discriminant function $\delta$ given by 
\begin{equation*}
\delta(u,v)=\left(lN-nL\right)^2-4(lM-mL)(mN-nM).
\end{equation*}

\begin{lemma}
At an elliptic umbilic point, we have that
 \begin{equation*}
 \det(\text{Hess}(  \delta) {(0,0)})=64 J^2,
\end{equation*}
while at an hyperbolic umbilic point,
\begin{eqnarray*}
	\det(\text{Hess}(  \delta) {(0,0)})=-64 J^2.
\end{eqnarray*}
\end{lemma}
\begin{proof}
Straightforward calculations.\qed
\end{proof}

Assuming that $J\ne 0$, the discriminant function has a Morse singularity at the origin; this is, $\delta$ is equivalent, by a change of coordinates in the source, to $u^2+v^2$ in the elliptic case and to $u^2-v^2$ in the hyperbolic case. In the elliptic case, the curvature lines are defined outside the isolated umbilic point, while in the hyperbolic case there exist two smooth curves crossing transversally the origin and the solutions of \eqref{eq_curv_lines2} are in the region where $\delta\geq0$.
The binary differential equation with discriminant function having Morse singularity at the origin of type $A_1^+$, resp. $A_1^-$ (in the Arnold's notation; for more details about $A_k^{\pm}$ singularities see \cite[Chapter 11]{Arnold1985}), was well studied in \cite{Bruce1995} and the conclusion of this part is referred to this paper.
\begin{figure}[h!!!]
\centering
	\includegraphics[width=10cm,clip]{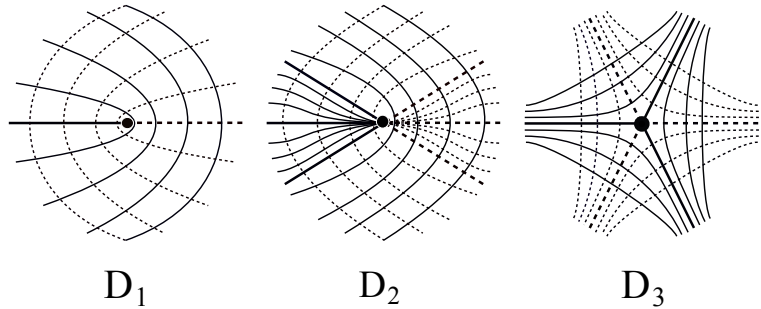}\\
	\caption{Affine lines of curvature near  an  affine umbilic point (elliptic case).}
	\label{fig1}
	\end{figure}

Consider the cubic polynomial 
\begin{equation*}	
	 p_{3e}= {  b_1}\,{k}^{3}+ \left( {  a_1}-2\,{  b_2} \right) {k}^{2} - \left(  { b_1}+2\,{  a_2} \right) k-{  a_1}.
\end{equation*}
in the elliptic case and 
\begin{equation*}
p_{3h}= { b_1}\,{k}^{3}+ \left( {  a_1}+2\,{  b_2} \right) {k}^{2}+\left( { b_1}+2\,{  a_2} \right) k+ {  a_1}.
\end{equation*}
in the hyperbolic case. 
The discriminants of these polynomials are denoted by $\Delta_e$ and $\Delta_h$, respectively.
%
The condition $\Delta_e\neq 0$, resp. $\Delta_h\neq 0$, means that $p_{3e}$, resp. $p_{3h}$, has no double roots. For the next two propositions, see Fig. 2 and 3 of \cite{Bruce1995}. See also \cite{Bruce1989} and \cite{GS-1982}.

\begin{proposition}\label{prop_A1+}
Consider an isolated elliptic affine umbilic point and assume $J\neq 0$ and $\Delta_e\neq 0$. Then the configuration of affine curvature lines are locally topological equivalent to the models presented in Fig. \ref{fig1}.  The affine configuration is completely determined by the 5-jet of the surface. More precisely:
\begin{itemize}
\item[(i)]  If $\Delta_e < 0$ then the affine umbilic point  is of type $D_1$, having a hyperbolic sector for each family of affine principal  lines. The binary differential equation of affine curvature lines is topologically equivalent to $vdv^2+2ududv-vdu^2=0$. See Fig. \ref{fig1}, left.
\item[(ii)]  If $\Delta_e > 0$ and $J <0$ then the affine umbilic point  is of type $D_2$, having a hyperbolic sector and a parabolic sector  for each family of affine principal  lines. The binary differential equation of affine curvature lines is topologically equivalent to $vdv^2+\frac 12 ududv-vdu^2=0$. See Fig. \ref{fig1}, center.
\item[(iii)]  If $\Delta_e > 0$ and $J >0$ then the affine umbilic point  is of type $D_3$, having three hyperbolic sectors for each family of affine principal  lines. The binary differential equation of affine curvature lines is topologically equivalent to $vdv^2-2ududv-vdu^2=0$. See Fig. \ref{fig1}, right.
\end{itemize}
%
\end{proposition}

\begin{proof} See  \cite{Bruce1989}, \cite{GS-1982}, \cite{GS-1991}. The configuration is established performing the  resolution of the binary differential equation using the Lie-Cartan approach.    The construction of the topological equivalence can be done by the method of canonical regions.  \qed
	\end{proof}

In the hyperbolic case, we must also consider the branches of the umbilic point, which are given by the zeros of 
the polynomial 
$$
p_{2h}=\left( -4\,  b_1^{2}+4\,  b_2^{2} \right) {k}^{2}+
\left( -8\,{  a_1}\,{  b_1}+8\,{  a_2}\,{  b_2} \right) k
-4\, a_1^{2}+4\,  a_2^{2}.
$$
The resultant of $p_{2h}$ and $p_{3h}$ is given by
\begin{equation*}
R=64\, \left( -{  b _1}+{  a _1}-{  a_2}+{  b_2} \right) ^{2}
\left( {  b_1}+{  a_1}+{  a_2}+{  b_2} \right) ^{2}
\left( {  a_1}\,{  b_2}-{  a_2}\,{  b_1} \right) ^{2}.
\end{equation*}
The condition $R\neq 0$ means that $p_{2h}$ and $p_{3h}$ have no common roots.
%
%
\begin{proposition}\label{prop_A1-}
Consider an isolated hyperbolic affine umbilic point and assume $J\neq 0$, $\Delta_h\neq 0$ and $R\neq 0$. Then the configurations of affine curvature lines are locally topological equivalent to the models presented in Fig. \ref{fig2}.  The affine configurations are completely determined by the 5-jet of the surface. More precisely:
\begin{itemize}
\item[(i)]  If $\Delta_h<0$ and $J<0$ then the affine umbilic point is of type $A_1$, having a hyperbolic sector for each family of affine principal  lines. The binary differential equation of affine curvature lines is topologically equivalent to $vdv^2+2ududv+vdu^2=0$. See $A_1$ Fig. \ref{fig2}. 
\item[(ii)] If $\Delta_h<0$ and $J>0$ then the affine umbilic point  is of type $A_2$, having a parabolic sector for each family of affine principal  lines. The binary differential equation of affine curvature lines is topologically equivalent to $vdv^2-\frac 12 ududv+vdu^2=0$. See $A_2$ Fig. \ref{fig2}.
\item[(iii)] If $\Delta_h>0$, $a_2+b_1>\left| a_1+b_2 \right|$ and $J>0$ then the affine umbilic point  is of type $A_3$, having one hyperbolic sectors and two parabolic sectors for each family of affine principal lines. The binary differential equation of affine curvature lines is topologically equivalent to $vdv^2-\frac 43ududv+vdu^2=0$. See $A_3$ Fig. \ref{fig2}.
\item[(iv)] If $\Delta_h>0$, $a_2+b_1>\left| a_1+b_2\right|$ and $J<0$ or $\Delta_h>0$, $\left|a_2+b_1\right|<\left| a_1+b_2\right|$ and $J>0$ then the affine umbilic point  is of type $A_4$, having two hyperbolic sectors and one parabolic sector for each family of affine principal  lines. The binary differential equation of affine curvature lines is topologically equivalent to $vdv^2+2(v-u)dudv+vdu^2=0$. See $A_4$ Fig. \ref{fig2}.
\item[(v)] If $\Delta_h>0$, $a_2+b_1<-\left|a_1+b_2\right|$  or $\Delta_h>0$, $\left|a_2+b_1\right|<\left|a_1+b_2\right|$ and $J<0$ then the affine umbilic point  is of type $A_5$, having three hyperbolic sectors for each family of affine principal lines. The binary differential equation of affine curvature lines is topologically equivalent to $vdv^2-4ududv+vdu^2=0$. See $A_5$ Fig. \ref{fig2}.
\end{itemize}
%
\end{proposition}
\begin{figure}[H]
	\centering
	\includegraphics[width=8.4cm,clip]{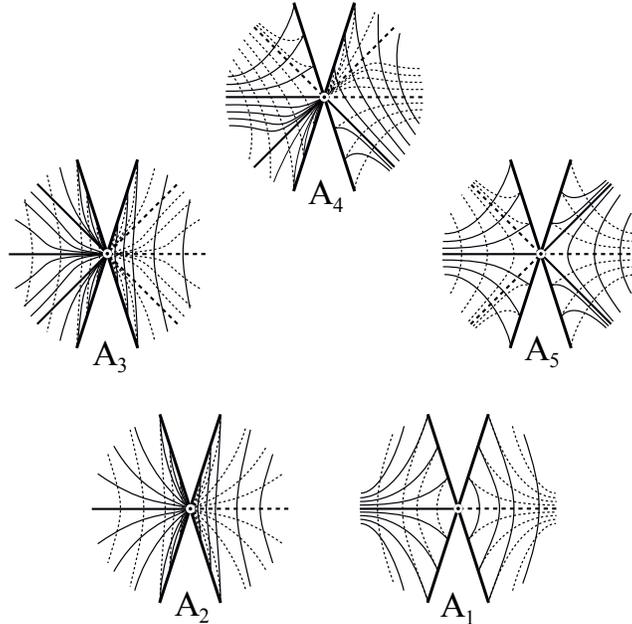}\\
	\caption{Affine lines of curvature near  an affine umbilic point (hyperbolic case).
	}\label{fig2}
\end{figure}
\begin{proof} See  \cite{Bruce1995}. The five local models are obtained performing a resolution of the binary differential equation in terms of hyperbolic singularities of local vector fields defined in the implicit surface of the differential equation. The local homeomorphism given the topological equivalence to the topological normal forms stated   can be constructed using the method of canonical regions, see \cite{GS-1982}.\qed
	\end{proof}

\section{Affine curvature lines near points with double eigenvalues of the affine shape operator}\label{section4}
%
In this section we consider the singularities of equation \eqref{eq_curv_lines2} when the affine shape operator has a double eigenvalue but is not a  hyperbolic umbilic point. Consider a surface $S\in\mathbb{R}^3$ parametrized by $X(u,v)=\left(u,v,h(u,v)\right)$. According  to \cite[Chapter 1]{Buchin}, in a small neighborhood of hyperbolic point $p=X(0,0)=(0,0)$, we can take the parametrization $(u,v,h(u,v))$, where 
\begin{eqnarray}\label{eq-nor_form_Buchin}
h(u,v) & = & uv+\frac{1}{6}\left(q_{30}u^3+q_{03}v^3\right)+\frac{1}{24}\left(q_{40}u^4+4q_{31}u^3v+6q_{22}u^2v^2+\right.\nonumber\\ 
& + & \left.4q_{13}uv^3+q_{04}v^4\right)+\frac{1}{120}\left(q_{50}u^5+5q_{41}u^4v+10q_{32}u^3v^2+\right.\nonumber\\
& + & \left.10q_{23}u^2v^3+5q_{14}uv^4+q_{05}v^5\right)+\frac{1}{720}\left(q_{60}u^6+6q_{51}u^5v+\right.\\
& + & \left.15q_{42}u^4v^2+20q_{33}u^3v^3+15q_{24}u^2v^4+6q_{15}uv^5+q_{06}v^6\right)+O(7).\nonumber
\end{eqnarray}
We will adopt this normal form instead of that used in the analysis of hyperbolic umbilic points, see Proposition \ref{Pick_norm_forms}, in order to simplify the calculations.

At $(0,0)$ it follows that $\xi_u(0,0)= -\frac{1}{2}(q_{22}-\frac{1}{2}q_{30}q_{03},q_{31})$ and $\xi_v(0,0)=-\frac{1}{2}(q_{13},q_{22}-\frac{1}{2}q_{30}q_{03}).$ Assuming $q_{13}\neq 0$ and $q_{31}=0$,  the double eigenvalue is  $-\frac{1}{2}(q_{22}-\frac{1}{2}q_{30}q_{03}) $ and  the unique eigenvector is $(1,0).$
Substituting \eqref{eq-nor_form_Buchin} in \eqref{eq_curv_lines2} we obtain that the differential equation of affine principal lines is given by
\begin{equation*}
A(u,v)du^2+B(u,v)dudv+C(u,v)dv^2=0,
\end{equation*}
where,
\begin{eqnarray*}
	A(u,v) & = & -32q_{31}+16\left(-2q_{03}q_{30}^2+7q_{22}q_{30}-32q_{41}\right)u+16\left(q_{03}q_{40}+\right.\\
	& + & \left.4q_{13}q_{30}-2q_{32}\right)v-16\left(3q_{03}q_{30}q_{40}+3q_{13}q_{30}^2-6q_{22}q_{40}-\right.\\
	& - & \left.5q_{30}q_{32}+q_{31}^2+q_{51}\right)u^2-16\left(10q_{03}q_{30}q_{31}+2q_{04}q_{30}^2-q_{03}q_{50}-\right.\\
	& - & \left.5q_{13}q_{40}-5q_{22}q_{31}-7q_{23}q_{30}+2q_{42}\right)uv+4\left(7q_{03}^2q_{30}^2-\right.\\
	& - & \left.32q_{03}q_{22}q_{30}+4q_{03}q_{41}+2q_{04}q_{40}+8q_{14}q_{30}+18q_{22}^2-4q_{33}\right)v^2+\\
	& + & O(3),
	\end{eqnarray*} 
	\begin{eqnarray*}
	B(u,v) & = & 32q_{30}q_{13}u-32q_{03}q_{31}v-16\left(4q_{03}q_{30}q_{31}+q_{04}q_{30}^2-q_{13}q_{40}+\right. \\
	& + & \left.q_{22}q_{31}-2q_{23}q_{30}\right)u^2-32(q_{03}q_{41}-q_{14}q_{30})uv+16\left(q_{03}^2q_{40}+\right. \\
	& + & \left.4q_{03}q_{13}q_{30}-2q_{03}q_{32}-q_{04}q_{31}+q_{13}q_{22}\right)v^2+O(3),
	\end{eqnarray*} 
	\begin{eqnarray*}
	C(u,v) & = & 32q_{13}-16\left(4q_{03}q_{31}+q_{04}q_{30}-2q_{23}\right)u+16\left(2q_{03}^2q_{30}-7q_{03}q_{22}+\right. \\
	& + & \left.2q_{14}\right)v-4\left(7q_{03}^2q_{30}^2-32q_{03}q_{22}q_{30}+8q_{03}q_{41}+2q_{04}q_{40}+\right. \\
	& + & \left.4q_{14}q_{30}+18q_{22}^2-4q_{33}\right)u^2+16\left(2q_{03}^2q_{40}+10q_{03}q_{13}q_{30}-\right.\\
	& - & \left.7q_{03}q_{32}-5q_{04}q_{31}-q_{05}q_{30}-5q_{13}q_{22}+2q_{24}\right)uv+16\left(3q_{03}^2q_{31}+\right. \\
	& + & \left.3q_{03}q_{04}q_{30}-5q_{03}q_{23}-6q_{04}q_{22}+q_{13}^2+q_{15}\right)v^2+O(3).
\end{eqnarray*} 
The  discriminant function    $\delta=B^2-4AC$ is given by
\begin{eqnarray*}\label{eq:delta}
	\delta(u,v) & = & 4096q_{31}q_{13}+2048\left(2q_{03}q_{13}q_{30}^2-4q_{03}q_{31}^2-q_{04}q_{30}q_{31}-\right.\\
	& - & \left.7q_{13}q_{22}q_{30}+2q_{13}q_{41}+2q_{23}q_{31}\right)u+2048\left(2q_{03}^2q_{30}q_{31}-\right.\\
	& - &\left.q_{03}q_{13}q_{40}-7q_{03}q_{22}q_{31}-4q_{13}^2q_{30}+2q_{13}q_{32}+2q_{14}q_{31}\right)v+O(2).
\end{eqnarray*}
Suppose $q_{31}=0$ and $q_{13}\neq0$, thus $p=(0,0)$ is a point in the discriminant set which is locally a smooth curve.

\begin{lemma}\label{lem:cusp} Let $q_{31}=0$,  $q_{13}(2q_{03}q_{30}^2-7q_{22}q_{30}+2q_{41} )\neq0$.  Then the configuration of affine curvature lines are locally topological equivalent near $p=(0,0)$ to the model  presented   in Fig. \ref{fig:cusp}. The binary differential equation of affine curvature lines  is topologically equivalent to the normal form $dv^2+u du^2=0.$
\end{lemma}

\begin{figure}[H]
\centering
	\includegraphics[scale=0.3,clip]{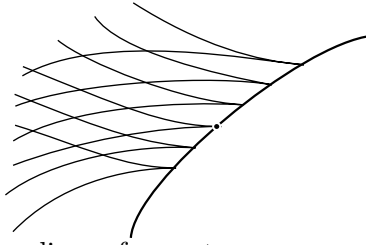}\\
	\caption{Affine lines of curvature near a point with double eigenvalue of the affine shape operator.\label{fig:cusp}
	} 
\end{figure}

\begin{proof} This  is a classical case, under the conditions stated in the lemma the discriminant set defined by $\delta=B^2-4AC=0$ is transversal to the unique affine principal direction at $\delta=0.$ Direct analysis shows that the solution  passing through $(0,0)$ is   parametrized by
	
	$$\aligned u(t)=& 512\,  q_{13}\, \left( 2\,q_{03}\,q_{30}^{2}-7\,  q_{22}\, 
	q_{30}+2\,  q_{41} \right) {t}^{2}+O(t^3)
	\\
	v(t)=&  \frac{16384}3 q_{13}\, \left( 2\,q_{03}\,q_{30}^{2}-7\,  q_{22}\, 
	q_{30}+2\,  q_{41} \right)^2 {t}^{3}+O(t^4)
	\endaligned $$
	Therefore, locally the solutions are curves having cuspidal type along the discriminant set $\delta=0.$
	
	The local homeomorphism of the topological equivalence  between the binary differential equation of affine curvature lines and the normal form stated can be constructed using the method of canonical regions, see \cite{GS-1982}.
\qed	
\end{proof}

\begin{lemma}\label{lem:double}
	Consider the parametrization given by \eqref{eq-nor_form_Buchin} and suppose that $q_{13}\ne 0$, $q_{31}=0$, $2q_{03}q_{30}^2-7q_{22}q_{30}+2q_{41}=0$. Let\\
{\small $\left(18q_{13}q_{30}^{2}+ \left(13q_{03}q_{40}-22q_{32}\right)  q_{30}-24  q_{22} q_{40}+4\ q_{51} \right)\left(  q_{ 03} q_{40} +4 q_{13} q_{30}-2q_{32} \right)\ne 0$.}\\
Then the  discriminant curve $\delta=0$  is regular near $p$, tangent to the double affine principal direction $(1,0)$  and the contact between the discriminant curve and the asymptotic line passing through $p$ and tangent to $(1,0)$ is quadratic.
\end{lemma}
%
\begin{proof}
	The asymptotic lines are solutions of  $ Ldu^2+2Mduv+Ndv^2=0$. In the conditions stated it is given by the implicit differential equation
	\begin{eqnarray*}
		H & = & \left(q_{30}u+\frac{1}{2}q_{40}u^2+\frac{1}{2}q_{22}v^2+O(3)\right)du^2+\left(2-u^2+2q_{22}uv+\right.\\
		& + & \left.\left(q_{13}-1\right)v^2+O(3)\right)dvdu+\left(q_{03}v+\frac{1}{2}q_{22}u^2+q_{13}uv+\frac{1}{2}q_{04}v^2+\right.\\
		& + & \left.O(3)\right)dv^2=0
	\end{eqnarray*}
	The asymptotic lines through $(0,0)$ are given by $(u,v_1(u))$ and $(u_2(v),v)$ with
	\begin{equation*}
	v_1(u)=-\frac{1}{4}q_{30}u^2-\frac{1}{12}q_{40}u^3+O(4)\quad\text{and}\quad u_2(v)=-\frac{1}{4}q_{03}v^2-\frac{1}{12}q_{04}v^3+O(4).
	\end{equation*}
	Under the conditions $q_{31}=0$ and $2q_{03}q_{30}^2-7q_{22}q_{30}+2q_{41}=0$,  the discriminant curve $\delta=0$    is locally defined by
	\begin{eqnarray*}
		\delta(u,v) & = & -2048q_{13}\left(q_{03}q_{40}+4q_{13}q_{30}-2q_{32}\right)v+1024q_{13}\left(6q_{03}q_{30}q_{40}+\right. \\
		& + & \left.7q_{13}q_{30}^2-12q_{22}q_{40}-10q_{30}q_{32}+2q_{51}\right)u^2+1024\left(q_{03}q_{04}q_{30}q_{40}+\right. \\
		& + & 8q_{04}q_{13}q_{30}^2-2q_{03}q_{13}q_{50}-2q_{03}q_{23}q_{40}-2q_{04}q_{30}q_{32}-10q_{13}^2q_{40}-\\
		& - & \left.22q_{13}q_{23}q_{30}+4q_{13}q_{42}+4q_{23}q_{32}\right)vu+-512\left(4q_{03}^3q_{30}q_{40}+\right. \\
		& + & 19q_{03}^2q_{13}q_{30}^2-14q_{03}^2q_{22}q_{40}-8q_{03}^2q_{30}q_{32}-74q_{03}q_{13}q_{22}q_{30}+\\
		& + & 4q_{03}q_{14}q_{40}+28q_{03}q_{22}q_{32}+2q_{04}q_{13}q_{40}+24q_{13}q_{14}q_{30}+\\
		& + & \left.18q_{13}q_{22}^2-4q_{13}q_{33}-8q_{14}q_{32}\right)v^2+O(3).
	\end{eqnarray*}
	Therefore by the Implicit Function Theorem it follows that for the curve $(u,\delta_1(u))$ we have:
	\begin{equation*}
	\delta_1(u)=\frac{1}{2}\left(\frac{6q_{03}q_{30}q_{40}+7q_{13}q_{30}^2-12q_{22}q_{40}-10q_{30}q_{32}+2q_{51}}{q_{03}q_{40}+4q_{13}q_{30}-2q_{32}}\right)u^2+O(3).
	\end{equation*}
	So the contact is quadratic when 
	$$\frac{6q_{03}q_{30}q_{40}+7q_{13}q_{30}^2-12q_{22}q_{40}-10q_{30}q_{32}+2q_{51}}{q_{03}q_{40}+4q_{13}q_{30}-2q_{32}}+\frac 12 q_{30}\ne 0.$$
The last condition is equivalent to\\	
	{\small  
		$\left(  18 q_{13}q_{30}^{2}+ \left( 13  q_{03}  q_{40}-22q_{32} \right)  q_{30}-24  q_{22} q_{40}+4\ q_{51} \right)\left(  q_{ 03} q_{40} +4 q_{13} q_{30}-2q_{32} \right)\ne 0$.}
	\qed
\end{proof}
\begin{proposition} \label{prop: sfn} In the conditions of Lemma \ref{lem:double}, the  configuration of  affine principal lines near a point with double eigenvalues of the affine shape operator is topologically equivalent to  the three  models shown in Fig. \ref{fig16}. The affine configuration is completely determined by the 6-jet of the surface.
	
	The binary differential equation of affine curvature lines is topologically equivalent to the following normal forms
	$$\aligned &{ \rm folded   \; saddle } \hskip 1cm dv^2-2u\, du dv  -vdu^2=0 \\
	&{ \rm folded   \; foci }\hskip 1cm \;\;\; \; dv^2-2u\, dudv +(2 u^2-v )du^2=0 \\
	&	{ \rm folded   \;node }\hskip 1cm\; \; \, dv^2 -2 u\,dudv +(\frac{14}9u^2-v)du^2=0  \endaligned$$
		
\end{proposition}
\begin{figure}[h]\label{fig:H}
\centering
	\includegraphics[scale=0.55]{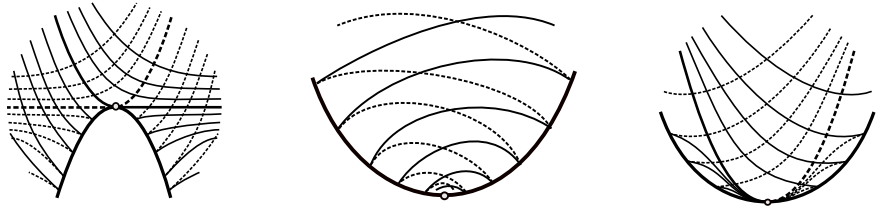}
	\caption{Affine curvature lines near a point with double eigenvalues of the affine shape operator. A folded saddle  (left), a folded foci (center) and folded node (right).}
	\label{fig16}
\end{figure}

\begin{proof} This result is well known in the literature, see \cite{livro_cidinha_farid_2015} and references therein. For convenience to the reader a sketch of proof is the following.
	Let $P=\frac{dv}{du}$ and
	\begin{equation*}
	\mathcal{F}(u,v,P)=A(u,v)+B(u,v)P+C(u,v)P^2.
	\end{equation*}
	Consider the Lie-Cartan vector field $\mathcal{X}(u,v,P)=\left(\mathcal{F}_P,P\mathcal{F}_P,-\left(\mathcal{F}_u+P\mathcal{F}_v\right)\right)$. We have $\mathcal{X}(0,0,0)=\left[0,0,0\right]$, thus $(0,0,0)$ is singular point of $\mathcal{X}$. The eigenvalues of  $D\mathcal{X}(0)$  are given by
	\begin{equation*}
	\lambda_i=-8\left(\left(q_{03}q_{40}+4q_{13}q_{30}-2a_{32}\right)+(-1)^i\sqrt{\Delta_\lambda}\right),\quad i=1,2,
	\end{equation*}
	where
	\begin{eqnarray*}
		\Delta_\lambda & = & q_{03}^2q_{40}^2+112q_{03}q_{13}q_{30}q_{40}+160q_{13}^2q_{30}^2-4q_{03}q_{32}q_{40}-192q_{13}q_{22}q_{40}-\\
		& - & 192q_{13}q_{30}q_{32}+32q_{13}q_{51}+4q_{32}^2.
	\end{eqnarray*}
	If $\Delta_\lambda<0$, the projection of the integral curves of $\mathcal{X}$ has a folded hyperbolic focus at the origin (see Fig. \ref{fig16}, center).  When $\Delta_\lambda>0$ we have
	\begin{equation*}
	\lambda_1\lambda_2=-512q_{13}\left(13q_{03}q_{30}q_{40}+18q_{13}q_{30}^2-24q_{22}q_{40}-22q_{30}q_{32}+4q_{51}\right)
	\end{equation*}
	and a we have a  hyperbolic folded  node ($	\lambda_1\lambda_2 >0$) or a hyperbolic folded saddle ($	\lambda_1\lambda_2<0 $). See   Fig. \ref{fig16} right and left, respectively.
	
	The construction of the  topological equivalence can be performed using the method of canonical regions, see  \cite{Garcia1999} and  \cite{GS-1982}. \qed
\end{proof}

%
%
%

\section{Affine principal curvature lines near   the  parabolic set}\label{section5}
In this section we study  the local behavior of  affine curvature lines in a small neighborhood of the parabolic set, which is assumed to be a regular curve.

\subsection{Differential equation of affine principal lines in a Monge chart}

In this subsection, we obtain the differential equation of the affine principal lines in a Monge chart.\\

Consider a smooth surface $S$ parametrized in a Monge chart $X(u,v)=(u,v,h(u,v))$. Recall  that   the affine  first   fundamental form is given by:

%
\begin{equation}
 I_a= \frac{1}{|h_{uu}h_{vv}-h{uv}^2|^{\frac{1}{4}}}\left(h_{uu}du^2 + 2  h_{uv} du dv+ h_{vv} dv^2\right).
\end{equation}
The affine normal vector $\xi$, defined by equations \eqref{Conormal_normal} and \eqref{eq:xicarta}, is given by $\xi=\left(\xi_1,\xi_2,\xi_3\right)$, where
\begin{eqnarray*}
\xi_1 & = & -\frac{1}{4}\frac{1}{\left(h_{uu}h_{vv}-h_{uv}^2\right)^{\frac{7}{4}}}\left( h_{uuu}h_{vv}^2+h_{uu}h_{vv}h_{uvv}-3h_{uv}h_{vv}h_{uuv} \right. \\
      & - & \left.h_{uu}h_{uv}h_{vvv}+2h_{uv}^2h_{uvv} \right),
\end{eqnarray*}
\begin{eqnarray*}
\xi_2 & = & -\frac{1}{4}\frac{1}{\left(  h_{uu}h_{vv}-h_{uv}^2\right)^{\frac{7}{4}}}\left( h_{uu}^2h_{uvv}+ h_{uu}h_{vv}h_{uuv}  -3 h_{uv} h_{uu }h_{uvv} \right. \\  
&-&\left. h_{uv}h_{vv}h_{uuu} +2h_{uv}^2 h_{uuv}   \right),
\end{eqnarray*}
\begin{eqnarray*}
\xi_3 & = & -\frac{1}{4}\frac{1}{\left(h_{uu}h_{vv}-h_{uv}^2\right)^{\frac{7}{4}}}\left(h_{u}h_{vv}^2h_{uuu}+h_{u}h_{uu}h_{vv}h_{vvu}-3h_{u}h_{uv}h_{vv}h_{uuv}-\right. \\
      & - & h_{v}h_{uv}h_{vv}h_{uuu}-3h_{v}h_{uu}h_{uv}h_{uvv}+2h_{v}h_{uv}^2h_{uuv}+h_{v}h_{uu}h_{vv}h_{uuv}+ \\
      & + & \left.h_{v}h_{uu}^2h_{vvv}-h_{u}h_{uu}h_{uv}h_{vvv}+2h_{u}h_{uv}^2h_{uvv}-4\left(h_{uu}h_{vv}-h_{uv}^2\right)^2\right).
\end{eqnarray*}
The coefficients of the third affine  fundamental form $ III_a=l du^2+2mdudv+ndv^2$, defined by equation \eqref{eq:xico},   are:
\begin{eqnarray*}
l & = & -\frac{1}{16}\frac{1}{\left(h_{uu}h_{vv}-h_{uv}^2\right)^{2}}\left(-4\left(h_{vv}h_{uuuu}-2h_{uv}h_{uuuv}\right)\left(h_{uu}h_{vv}-h_{uv}^2\right)\right.-\\
  & - & 4h_{uu}\left(h_{uu}h_{vv}-h_{uv}^2\right)h_{uuvv}+7h_{vv}^2h_{uuu}^2+3h_{uu}^2h_{uvv}^2+\\
	& + & (-28h_{uuv}h_{uv}h_{vv}+2(h_{uu}h_{vv}+8h_{uv}^2)h_{uvv}-4h_{vvv}h_{uu}h_{uv})h_{uuu}+\\
	& + & \left.12(h_{uu}h_{vv}+h_{uv}^2)h_{uuv}^2+4(h_{uu}^2h_{vvv}-6h_{uu}h_{uv}h_{uvv})h_{uuv}\right),
\end{eqnarray*}
\begin{eqnarray*}
m & = & -\frac{1}{16}\frac{1}{\left(h_{uu}h_{vv}-h_{uv}^2\right)^{2}}\left(-4\left(h_{vv}h_{uuuv}-2h_{uv}h_{uuvv}\right)\left(h_{uu}h_{vv}-h_{uv}^2\right)+\right.\\
  & + & (7h_{vv}^2h_{uuv}-10h_{uv}h_{vv}h_{uvv}+(-h_{uu}h_{vv}+4h_{uv}^2)h_{vvv})h_{uuu}- \\
  & - & 4\left(h_{uu}h_{vv}-h_{uv}^2\right)h_{uvvv}-18h_{uuv}^2h_{uv}h_{vv}+7h_{uvv}h_{vvv}h_{uu}^2+\\
	& + & \left.((15h_{uu}h_{vv}+24h_{uv}^2)h_{uvv}-10h_{uu}h_{uv}h_{vvv})h_{uuv}-18h_{uvv}^2h_{uu}h_{uv}\right),
\end{eqnarray*}
\begin{eqnarray*}
n & = & -\frac{1}{16}\frac{1}{\left(h_{uu}h_{vv}-h_{uv}^2\right)^{2}}\left(-4\left(h_{vv}h_{uuvv}-2h_{uv}h_{uvvv}\right)\left(h_{uu}h_{vv}-h_{uv}^2\right)-\right.\\
  & - & 4h_{uu}h_{vvvv}\left(h_{uu}h_{vv}-h_{uv}^2\right)+4(-h_{uv}h_{vv}h_{vvv}+h_{uvv}h_{vv}^2)h_{uuu}+\\
	& + & 3h_{uuv}^2h_{vv}^2+2(-12h_{uv}h_{vv}h_{uvv}+(h_{uu}h_{vv}+8h_{uv}^2)h_{vvv})h_{uuv}+\\
	& + & 12(h_{uu}h_{vv}+h_{uv}^2)h_{uvv}^2-28h_{uvv}h_{vvv}h_{uu}h_{uv}+7h_{vvv}^2h_{uu}^2.
\end{eqnarray*}
Thus, the equation of the  affine curvature lines \eqref{eq_curv_lines2} is given by
\begin{equation}\label{bde_parab}
A(u,v)du^2+B(u,v)dudv+C(u,v)dv^2=0,
\end{equation}
where
\begin{eqnarray*}
A(u,v) & = & \frac{1}{4}\frac{1}{\left(h_{uu}h_{vv}-h_{uv}^2\right)^2}\left(-4h_{uv}h_{vv}h_{uuuu}\left(h_{uu}h_{vv}-h_{uv}^2\right)+\right.\\
       & + & 4\left(h_{uu}h_{vv}-h_{uv}^2\right)\left(\left(h_{uu}h_{vv}+h_{uv}^2\right)h_{uuuv}-3h_{uu}h_{uv}h_{uuvv}\right)+\\
			 & + & 4h_{uu}^2h_{uvvv}\left(h_{uu}h_{vv}-h_{uv}^2\right)+7h_{uv}h_{vv}^2h_{uuu}^2+\\
			 & + & \left(-7h_{vv}h_{uuv}(h_{uu}h_{vv}+4h_{uv}^2)+4h_{uv}h_{uvv}(3h_{uu}h_{vv}+4h_{uv}^2)+\right.\\
			 & + & \left.h_{uu}h_{vvv}(h_{uu}h_{vv}-8h_{uv}^2)\right)h_{uuu}+6h_{uv}h_{uuv}^2(5h_{uu}h_{vv}+2h_{uv}^2)+\\
			 & + & (-3h_{uu}h_{uvv}(5h_{uu}h_{vv}+16h_{uv}^2)+14h_{uu}^2h_{uv}h_{vvv})h_{uuv}-\\
			 & - & \left.7h_{uu}^3h_{vvv}h_{uvv}+21h_{uu}^2h_{uv}h_{uvv}^2\right),
\end{eqnarray*}
\begin{eqnarray*}
B(u,v) & = & \frac{1}{4}\frac{1}{\left(h_{uu}h_{vv}-h_{uv}^2\right)^2}\left(-4h_{vv}^2h_{uuuu}\left(h_{uu}h_{vv}-h_{uv}^2\right)+\right.\\
       & + & 8h_{uv}\left(h_{vv}h_{uuuv}-h_{uu}h_{uvvv}\right)\left(h_{uu}h_{vv}-h_{uv}^2\right)+7h_{vv}^3h_{uuu}^2+\\
			 & + & 4h_{uu}^2h_{vvvv}\left(h_{uu}h_{vv}-h_{uv}^2\right)+3h_{vv}h_{uuv}^2(3h_{uu}h_{vv}+4h_{uv}^2)-\\
			 & - & 2h_{vv}(14h_{uv}h_{vv}h_{uuv}+h_{uvv}\left(h_{uu}h_{vv}-8h_{uv}^2\right))h_{uuu}+\\
			 & + & 2h_{uu}h_{uuv}h_{vvv}(h_{uu}h_{vv}-8h_{uv}^2)+28h_{uu}^2h_{uv}h_{vvv}h_{uvv}-\\
			 & - & \left.3h_{uu}h_{uvv}^2(3h_{uu}h_{vv}+4h_{uv}^2)-7h_{uu}^3h_{vvv}^2\right),
\end{eqnarray*}
\begin{eqnarray*}
C(u,v) & = & \frac{1}{4}\frac{1}{\left(h_{uu}h_{vv}-h_{uv}^2\right)^2}\left(-4h_{vv}^2h_{uuuv}\left(h_{uu}h_{vv}-h_{uv}^2\right)-\right.\\
       & - & 4\left(h_{uu}h_{vv}-h_{uv}^2\right)(h_{uvvv}\left(h_{uu}h_{vv}+2h_{uv}^2\right)-3h_{uv}h_{vv}h_{uuvv})+\\
			 & + & 4h_{uu}h_{uv}h_{vvvv}\left(h_{uu}h_{vv}-h_{uv}^2\right)-21h_{uv}h_{vv}^2h_{uuv}^2+\\
			 & + & h_{vv}h_{uuu}(7h_{vv}^2h_{uuv}-14h_{uv}h_{vv}h_{uvv}-h_{vvv}\left(h_{uu}h_{vv}-8h_{uv}^2\right))-\\
			 & - & h_{uv}h_{uvv}^2(5h_{uu}h_{vv}+2h_{uv}^2)+7h_{uu}h_{uvv}h_{vvv}(h_{uu}h_{vv}+4h_{uv}^2)-\\
			 & - & 7h_{uu}^2h_{uv}h_{vvv}^2+h_{uuv}\left(3h_{vv}h_{uvv}(5h_{uu}h_{vv}+16h_{uv}^2)-\right.\\
			 & - & \left.\left.4h_{uv}h_{vvv}(3h_{uu}h_{vv}+4h_{uv}^2)\right)\right).
\end{eqnarray*}
Note that \eqref{bde_parab} is a homogeneous equation, thus we can extend the binary differential equation of affine curvature lines  to the   parabolic set, defined by $ h_{uu}h_{vv}-h_{uv}^2 =0$, as the numerator of \eqref{bde_parab}, written as
\begin{equation}\label{numer_bde_parab}
\bar{A}(u,v)du^2+\bar{B}(u,v)dudv+\bar{C}(u,v)dv^2=0.
\end{equation}
The  parabolic set is  generically, see \cite[Chapter 6]{livro_cidinha_farid_2015},  formed of smooth curves and it induces a natural decomposition on a surface: the elliptic region (where $K_e>0$) and the hyperbolic region (where $K_e<0$) having the parabolic set as common boundary. 


\subsection{ Affine principal curvature lines near an ordinary parabolic point}

The main goal of this subsection is to describe the local behavior of the affine principal lines near an ordinary parabolic points, i.e., when the double asymptotic direction given in the Monge chart by $(h_{vv},-h_{uv})=(-h_{uv},h_{uu})$ is transversal to the regular curve of parabolic points. We observe that, at such parabolic points, the Euclidean principal directions are both transversal to the parabolic curve. For affine principal lines we have the following.
\begin{proposition}\label{prop_pslca}
At an ordinary parabolic point, the  parabolic set is an affine curvature line.
\end{proposition}
\begin{proof}
We must prove that the parabolic set satisfies equation \eqref{numer_bde_parab}.
Let $p\in S$ be  parabolic point, i.e., $\left(h_{uu}h_{vv}-h_{uv}^2\right)(p)=0$. The tangent vector to the parabolic set at $p$ is given by $\left(du_1,dv_1\right)$ where
\begin{eqnarray*}
du_1 & = & -h_{uu}h_{vvv}-h_{uuv}h_{vv}+2h_{uv}h_{uvv}, \\
dv_1 & = & h_{uu}h_{uvv}+h_{uuu}h_{vv}-2h_{uuv}h_{uv}.
\end{eqnarray*}
Evaluating equation \eqref{numer_bde_parab} it follows, corroborated by algebraic symbolic calculations, that
$$ \bar{A}(u,v)du_1^2+\bar{B}(u,v)du_1dv_1+\bar{C}(u,v)dv_1^2= (h_{uu}h_{vv}-h_{uv}^2)H(h,dh,d^2h,d^3h,d^4h),$$
where $H$ is a polynomial in the partial derivatives of $h$ up to order four, which completes the proof.
\qed
\end{proof}
%
Let now $X(u,v)=(u,v,h(u,v))$ be a parametrization around a parabolic point $p=X(0,0)$ where,
\begin{eqnarray}\label{eq-param-parab}
h(u,v) & = & \frac{k}{2}v^2+\frac{1}{6}\left(q_{30}u^3+3q_{21}u^2v+3q_{12}uv^2+q_{03}v^3\right)+\frac{1}{24}\left(q_{40}u^4+\right.\nonumber\\ 
       & + & \left.4q_{31}u^3v+6q_{22}u^2v^2+4q_{13}uv^3+q_{04}v^4\right)+\frac{1}{120}\left(q_{50}u^5+\right.\\
			 & + & \left.5q_{41}u^4v+10q_{32}u^3v^2+10q_{23}u^2v^3+5q_{14}uv^4+q_{05}v^5\right)+O(6).\nonumber
\end{eqnarray}

Replacing \eqref{eq-param-parab} in \eqref{numer_bde_parab} we obtain, at    the origin, that:
{\small  
\begin{equation}\label{numer_bde_parab_monge}
	\aligned 
\bar{A}(u,v)&du^2+\bar{B}(u,v)dudv+\bar{C}(u,v)dv^2=0,\\
\bar{A}(u,v)=&7\,{k}^{2}  q_{30}\, \left(  q_{12}\,  q_{30}-  q_{21}^{2}
\right) v
+O(2)
 \\
\bar{B}(u,v)=& 7\,{k}^{3} q_{30}^{2}+{k}^{2}q_{30}\, \left( 10\,k  q_{40}+19
\,q_{12}\,q_{30}-19\,q_{21}^{2} \right) u\\
-&   {k}^{2} \left( 4
\,kq_{21} \,q_{40}-14\,kq_{30}\,q_{31}-21\,q_{03}\,q_{30}^{2}+30\,q_{12}\,q_{21}\,q_{30}-9\,q_{21}^{3}
\right) v+O(2)
  \\
\bar{C}(u,v)= & 7\,{k}^{3}  q_{21}\,  q_{30}+{k}^{2} \left( 7\,k  q_{21}\,  q_{40}
+3\,k  q_{30}\, q_{31}-  q_{03}\,  q_{30}^{2}+22\, q_{12}\, 
 q_{21}\,  q_{30}-21\,  q_{21}^{3} \right) u \\
+&{k}^{2} \left( 3\,k  q_{21}\,  q_{31}+7\,k  q_{22}\,  q_{30}+20\,  q_{03}\,  q_{21}
\,  q_{30}-14\, q_{12}^{2}  q_{30}-6\,  q_{12}\,  q_{21}^{2}
\right) v+O(2)
&  
\endaligned
\end{equation}
}
In particular \eqref{numer_bde_parab_monge} evaluated at the origin gives $7k^3q_{30}(q_{30}du+q_{21}dv)dv=0$. 
The condition for $p$ to be an ordinary parabolic point is exactly $q_{30}\neq 0$.
Therefore we conclude that:
\begin{lemma}
At an ordinary parabolic point, the affine principal lines are transversal. The local behavior of the affine principal lines is as shown in Fig. \ref{fig:psapl}. The affine configuration is topologically equivalent to the topological normal form $ v\, du dv=0$ and it is completely determined by the 4-jet of the surface.
\end{lemma}
	\begin{figure}[H]
	\centering
	\includegraphics[width=4cm,clip]{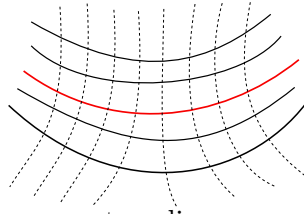}\\
	\caption{ \label{fig:psapl} Affine curvature lines near ordinary parabolic points. The parabolic arc is an affine principal line.}
\end{figure}
\subsection{ Affine principal curvature lines near a Gauss cusp point}

The main goal of this subsection is to describe the local behavior of affine principal lines near a Gauss cusp point.
The  origin ($p=X(0,0)$) is a Gauss cusp point when $q_{30}=0$ and $3q_{21}^2-kq_{40}\neq0$ (see \cite[Chapter 6]{livro_cidinha_farid_2015}).\\

In this case, 
the discriminant function $\delta=\bar{B}^2-4\bar{A}\bar{C} $ associated to equation \eqref{numer_bde_parab} is given by
{\small 
\begin{eqnarray*}
	\delta(u,v) & = & 16k^4q_{21}^2\left(4kq_{40}-9q_{21}^2\right)^2v^2+Q_{12}uv^2+Q_{03}v^3\\
	&-&16k^3q_{21}\left(40k^3q_{40}^3-554k^2q_{21}^2q_{40}^2+2211kq_{21}^4q_{40}-2736q_{21}^6\right)u^2v\\
	& + & 4k^2\left(4608q_{21}^8-6240kq_{21}^6q_{40}+3513k^2q_{21}^4q_{40}^2-948k^3q_{21}^2q_{40}^3+100k^4q_{40}^4\right)u^4\\
	& + & Q_{31}u^3v+Q_{22}u^2v^2+Q_{13}uv^3+Q_{04}v^4+O(5),
\end{eqnarray*}
}
where $Q_{ij}$ is function of $k,q_{21},q_{12},q_{03},q_{40},q_{31},q_{22},q_{13},q_{04}$ and the coefficients of order five and six.

The function $\delta$ has principal part (accordingly to the Newton's polygon) given by:
{\small 
$$\aligned \delta_p(u,v)=&\frac{1}{4}k^{2}\left(100k^{4}q_{40}^{4}-948k^{3}q_{21}^{2}q_{40}^{3}+3513k^{2}q_{21}^{4}q_{40}^{2}-6240kq_{21}^{6}q_{40}+\right.\\
+&\left.4608q_{21}^{8}\right)u^{4}+q_{21}k^{3}\left(-40k^{3}q_{40}^{3}+554k^{2}q_{21}^{2}q_{40}^{2}-2211k q_{21}^{4}q_{40}+\right.\\
+&\left.2736q_{21}^{6}\right)u^{2}v+q_{21}^{2}{k}^{4}\left(9q_{21}^{2}-4kq_{40}\right)^{2}v^{2}.\endaligned
$$
}
It is worthwhile to mention that $\delta_p$ does not depend on $q_{12},q_{03},q_{31},q_{22},q_{13}$ and $q_{04}$.

We can check using symbolic algebraic manipulators that the determinant of the Hessian of $\delta_P(u,v)=\delta_p(\sqrt{u},v)$ is 
$$
det(Hess(\delta_P)(0,0)= 2^{14} 7^3 k^6q_{21}^4(kq_{40}-4q_{21}^2)(kq_{40}-3q_{21}^2)^4.
$$
Thus $\delta$ has a $A^{\pm}_{3}$ singularity at the origin if, and only if, 
\begin{equation}\label{eq:A3pm}
9q_{21}^2-4kq_{40}\neq0 \quad\text{and}\quad k^6q_{21}^4(kq_{40}-4q_{21}^2)(kq_{40}-3q_{21}^2)^4\neq 0,
\end{equation}
(for more details about $A_k^{\pm}$ singularities see \cite[Chapter 11]{Arnold1985}).\\
In fact, $\delta$ has a $A^{+}_{3}$ singularity if $kq_{40}-4q_{21}^2>0$ and $A^{-}_{3}$ singularity if $kq_{40}-4q_{21}^2<0$.
%
%
%
%
Equation \eqref{eq:A3pm} means that $\delta_P(u,v)$ is a nondegenerate quadratic form.

When $\delta$ has a $A^{\pm}_{3}$ at the origin, by composition of diffeomorphism in the source, $\delta$ can be rewritten locally as $v^2\pm u^4$, see \cite[Chapter 11, page 188]{Arnold1985}. When $\delta$ has a $A^{-}_3$ singularity at the origin, the double $\xi$-direction set is locally homeomorphic to $v^2-u^4=0$,  a pair of parabolas. If $\delta$ has an $A^{+}_3$ singularity at the origin, the double $\xi$-direction set is locally homeomorphic to $v^2+u^4=0$, an isolated point.

%
\begin{remark}\label{rem_deg}
	\begin{enumerate}
		\item Note that if $kq_{21}\left(4kq_{40}-9q_{21}^2\right)=0$, then $\delta$ has a singularity of corank 2 at the origin. 
		\item If $kq_{40}-3q_{21}^2=0$, $p$ is more degenerate than a Gauss cusp point. 
		\item Finally, if $kq_{21}\left(4kq_{40}-9q_{21}^2\right)\neq 0$ and $kq_{40}-4q_{21}^2=0$, then $\delta$ has a singularity of type $A_k$, $k>3$.
	\end{enumerate}
\end{remark}

When $\delta$ has a $A^{-}_3$ singularity at the origin, the curves that form the double $\xi$-direction set are given by $v=\alpha_i u^2+O(3)$, $i=1,2,$ where $\alpha_i$ are the roots of the quadratic equation $P_{22} \alpha^2+P_{12}\alpha+P_{11}=0$, where 
\begin{eqnarray*}
P_{22}  & = & k^4q_{21}^2\left(4kq_{40}-9q_{21}^2\right)^2, \\
P_{12}  & = & -k^3q_{21}\left(40k^3q_{40}^3-554k^2q_{21}^2q_{40}^2+2211kq_{21}^4q_{40}-2736q_{21}^6\right),\\
P_{11}& = & \frac{1}{4}k^2\left(100k^4q_{40}^4-948k^3q_{21}^2q_{40}^3+3513k^2q_{21}^4q_{40}^2-6240kq_{21}^6q_{40}+\right.\\
    & + & \left.4608q_{21}^8\right).
\end{eqnarray*}

In fact, $\delta_p(u,\alpha_i u^2)=0 $ and  $\delta(u,\alpha_i u^2+\cdots)=O(5). $

\begin{proposition}\label{aff_curv-lin_cusp_G}
The differential equation of the affine curvature lines near a Gauss cusp point can be reduced to the following normal forms up to coefficients of higher order. \\
\begin{equation}
\label{eq:normal_form_pa}\aligned 
-u^3du^2+& 2\left(b_{01}v+\bar{b}_{20}u^2 
\right)dudv+udv^2=0,\;\; \text{\rm if } \;\; kq_{40}-4q_{21}^2>0 ;\\
 u^3du^2+&2\left(b_{01}v+b_{20}u^2
 \right)dudv+udv^2=0,\;\; \text{\rm if } \;\;  kq_{40}-4q_{21}^2<0. 
 \endaligned
\end{equation}

where $b_{01}(q_{21},q_{40})=-\frac{1}{14}\frac{9q_{21}^2-4q_{40}k}{3q_{21}^2-q_{40}k}\ne 0$,
\begin{equation*}	\aligned
b_{20}(q_{21},q_{40}) 
=&-\frac{\sqrt{2}}{4}\frac{4b_{01}^2+13b_{01}+4}{\sqrt{\left(2b_{01}+1\right)\left(7b_{01}+2\right)}}.
	\endaligned  
\end{equation*}
\text{and}
\begin{equation*}
	\aligned
\bar{b}_{20}(q_{21},q_{40}) 
=&-\frac{\sqrt{2}}{4}\frac{4b_{01}^2+13b_{01}+4}{\sqrt{-\left(2b_{01}+1\right)\left(7b_{01}+2\right)}}.
\endaligned
\end{equation*}
\end{proposition}
\begin{proof}
See Appendix.	\qed
\end{proof}

Note that $b_{01}$ and $b_{20} (\bar{b}_{20})$ do not depend on $q_{12},q_{03},q_{31},q_{22},q_{13}$ and $q_{04}$.
%
\begin{remark}
The BDE's with function discriminant having an $A_3^{\pm}$ singularity  was studied in \cite{Tari2007} where the author was motived by the differential geometry in  the neighborhood of a cross-cap (also called Whitney umbrella). The behavior of the Euclidean curvature lines near a geometric cross-cap, was also studied in \cite{Garcia2000}.
\end{remark}
The next is the main result of this section.\\ 
\begin{theorem}\label{Thm_confg_princ_affins}
Let $X:\mathbb{R}^2\rightarrow\mathbb{R}^3$ be a local parametrization of the surface $S$ in a small neighborhood of a   Gauss cusp point. Consider $b_{01}$ as in the Proposition \ref{aff_curv-lin_cusp_G}. Then the   configuration of the affine curvature lines is locally topologically equivalent to the models listed below. 
\begin{enumerate}
	\item[1.] If $kq_{40}-4q_{21}^2>0$, the affine curvature lines are locally topologically equivalent to the model presented  in the Fig. \ref{fig7}. (In this case we have $-\frac{1}{2}<b_{01}<-\frac{2}{7}$.)
		%
	The parabolic set is a regular curve and discriminant is an isolated point,   the origin.
	The binary differential equation  of the affine curvature lines  is topologically
	equivalent to
$$- u^3 du^2 - \frac 12 v\, du dv +udv^2=0.$$
	\begin{figure}[H]
		\centering
		\includegraphics[scale=0.60]{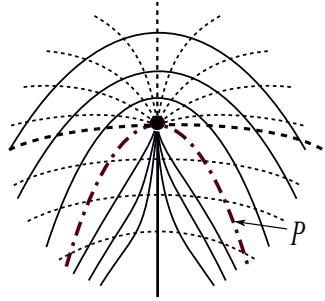}\\
		\caption{Affine curvature lines near   a Gauss cusp point  when the discriminant $\delta$ has an $A_3^{+}$ singularity at the origin. 
			The curve plotted by dashes and points corresponds to the parabolic set. \label{fig7}}
		
	\end{figure}	
	\item[2.] When $kq_{40}-4q_{21}^2<0$ we have five distinct topological models:
	%
	%
		\begin{enumerate}
			\item[i)] \;If $b_{01}<-\frac{1}{4}\left(5+\sqrt{21}\right)$, Fig. \ref{fig8}, {\rm (}$R_1${\rm  )};
			
			\item[ii)]\; If $-\frac{1}{4}\left(5+\sqrt{21}\right)<b_{01}<-1$, Fig.\ref{fig8}, {\rm (}$R_2${\rm  )};
			
			\item[iii)]\; If $-1<b_{01}<-\frac{1}{2}$ or $-\frac{2}{7}<b_{01}< \frac{1}{4}\left( \sqrt{21}-5\right)$, Fig. \ref{fig8},  {\rm (}$R_3${\rm  )};
			\item[iv)]\; If $  \frac{1}{4}\left( \sqrt{21}-5\right)<b_{01}<0$, Fig. \ref{fig8},  {\rm (}$R_4${\rm  )}; 
			\item[v)] \; If $b_{01}>0$ and $b_{01}\neq\frac{1}{4}$, Fig. \ref{fig8},  {\rm (}$R_5${\rm  )}.
		\end{enumerate}
\end{enumerate}

The discriminant is formed by two regular curves having a quadratic contact at the origin.

The binary differential equation  of the affine curvature lines  is topologically
 equivalent to one of the following topological normal forms, i.e., the affine configuration is topologically equivalent to one of the normal forms below.
%
%
%
%
%
%
%

%
 
	%
	\begin{eqnarray*}
		\left(R_1\right) &   & u^3du^2- 2\left(3v+  \frac{9}{250}  u^2\right)dudv+udv^2=0. \\ 
		\left(R_2\right) &   & u^3du^2+2 \left(-2v+  \frac{7}{20}  u^2\right)dudv+udv^2=0. \\ 
		\left(R_3\right) &   & u^3du^2+ 2\left(-\frac 34 v+  \frac{97}{100}u^2\right)dudv+udv^2=0. \\ 
		\left(R_4\right) &   & u^3du^2- 2\left(\frac{1}{20}v+ \frac{19}{20}  u^2\right)dudv+udv^2=0. \\ 
		\left(R_5\right) &   & u^3du^2+ 2\left( v-  \frac{71}{50} u^2\right)dudv+udv^2=0. 
	\end{eqnarray*}
 
 The   topological models are completely determined by the 5-jet  of the surface.
\begin{figure}[H]
	\centering
	  \includegraphics[scale=0.67]{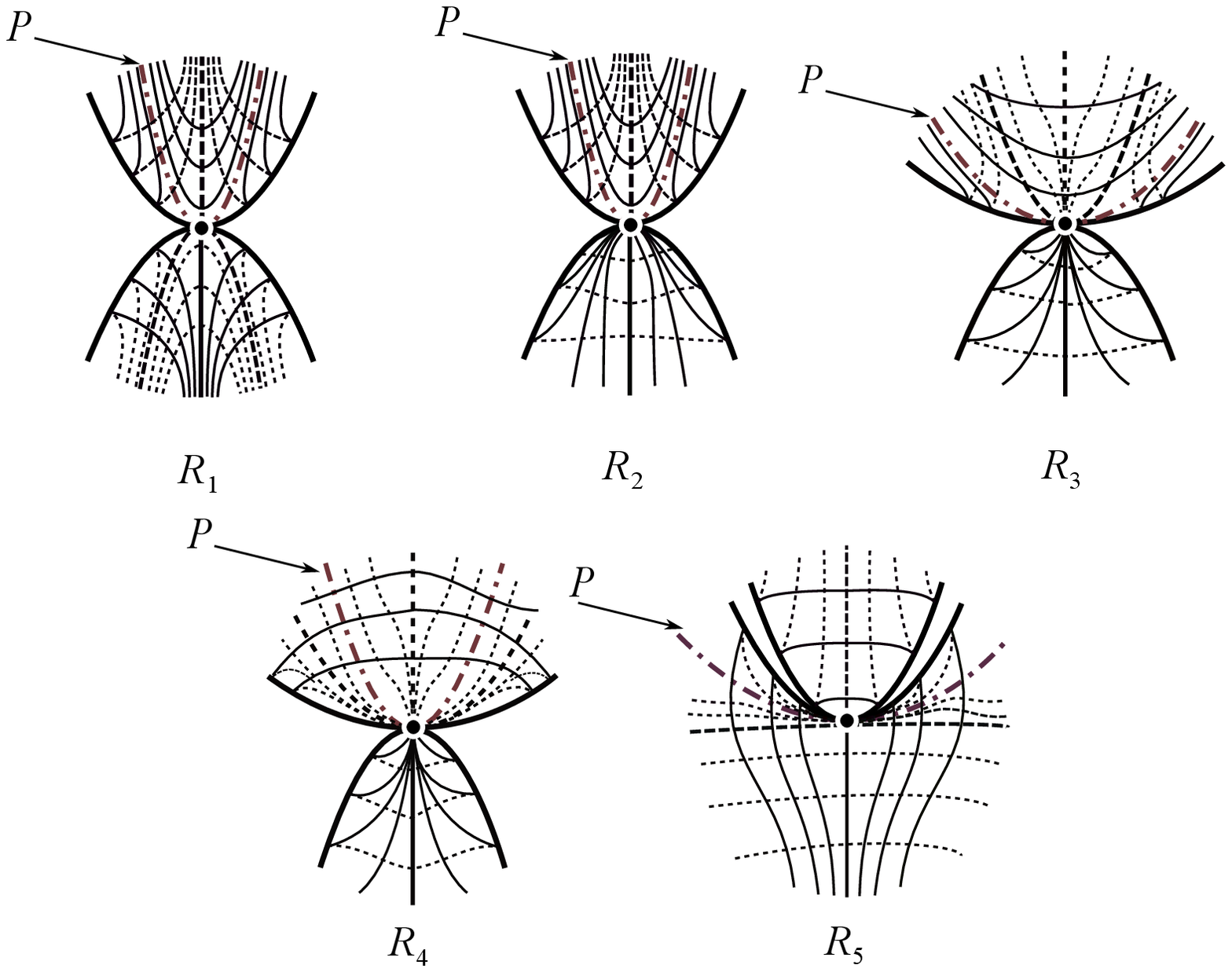}
\caption{Affine curvature lines near  a  Gauss cusp point  when the discriminant has a $A_3^{-}$ singularity at the origin. The curve plotted  by dashes and points corresponds to the parabolic set. \label{fig8}}
\end{figure}
%
\end{theorem}
%
\begin{proof}
We start considering the equation
\begin{equation}\label{BDE_N_F}
\pm u^3du^2+2(b_{01}v+b_{20}u^2)dudv+udv^2=0
\end{equation}
and naturally, we have two cases to analyze. Note that there are three invariant curves (separatrices), one has a vertical tangent, and  the remaining two have a horizontal tangent and are given by $v=\left(P_i^{\pm}\right)u^2+O(3)$, $i=1,2$, where $P_i^{\pm}$ are the roots of the equation 
\begin{equation}\label{eq_Par_Inv_A3+}
4\left(b_{01}+1\right)P^2+4b_{20}P\pm1=0.
\end{equation}
Consider the   weighted polar blowing-up in the equation \eqref{BDE_N_F} given by
\begin{equation*}
\varphi(t,r)=\left(r\cos(t),r^2\sin(t)\right)\quad\text{$r\geq0$ and $0<t<2\pi$}.
\end{equation*}
Note that via the application $\varphi$, for each angle $t_0$ corresponds the curve $v=\frac{\sin(t_0)}{\cos^2(t_0)}u^2+O(3)$, $t_0\neq\frac{\pi}{2},\frac{3\pi}{2}$ and the relative position of the curves depends on the sign of $\sin(t_0)$. We will use this correspondence conveniently.\\

\item[1.] When $kq_{40}-4q_{21}^2>0$, the new BDE (binary differential equation) in the variables $r,t$, after dividing by $r^3$ is given by
\begin{equation}\label{bde_r-t}
A(t)dr^2+2rB(t)drdt+r^2C(t)dt^2=0,
\end{equation}
where
\begin{eqnarray*}
A(t) & = & -\cos(t)\left(\cos^4(t)-4b_{20}\cos^2(t)\sin(t)-4(b_{01}+1)\sin^2(t)\right), \\
B(t) & = & \cos^2(t)\sin(t)(\cos^2(t)+2)+\left(\cos^2(t)b_{20}+\sin(t)b_{01}\right)(3\cos^2(t)-2),\\
C(t) & = & \cos(t)\left(\cos^4(t)-2b_{20}\cos^2(t)\sin(t)-2b_{01}\sin^2(t)\right).
\end{eqnarray*}
The singular points of the BDE \eqref{bde_r-t} are given by $r=0$ and the solutions of $A(t)=0$. Therefore the BDE \eqref{bde_r-t} has six hyperbolic singularities in the interval $[0,2\pi]$ which are given by $t=\frac{\pi}{2},\frac{3\pi}{2}$ and the roots of the quadratic equation 
\begin{equation}\label{eq_pi-}
P_i^{-}\sin^2(t)+\sin(t)-P^{-}_i=0,
\end{equation}
where $P_i^-$ is a solution of the equation \eqref{eq_Par_Inv_A3+}. Note that the equation \eqref{eq_pi-} is obtained directly via blow-down. Replacing in the equation $A(t)=0$, $P=\frac{\sin(t)}{\cos^2(t)}$ we obtain \eqref{eq_Par_Inv_A3+} for $\cos(t)\neq0$.\\
Consider now the vector fields defined by kernel of the differential forms
\begin{equation}\label{eq_VF_Y+}
Y_i=r\left(-B(t)+(-1)^i\sqrt{B^2(t)-A(t)C(t)}\right)dr+A(t)dt,\quad i=1,2.
\end{equation}
Outside their singularities, the vector fields $Y_i$ span the line fields associate to \eqref{bde_r-t} and they have the same singularities. We work in this case with $b_{01}$ and the expresion of $\bar{b}_{20}$ in terms of $b_{01}$, which is given by
\begin{equation}\label{eq_b20A3-}
\bar{b}_{20}=-\frac{\sqrt{2}}{4}\frac{4b_{01}^2+13b_{01}+4}{\sqrt{-\left(2b_{01}+1\right)\left(7b_{01}+2\right)}}.
\end{equation}
By equation \eqref{eq_b20A3-} we have directly $-\frac{1}{2}<b_{01}<-\frac{2}{7}$. The eigenvalues of the linearization of the vector field \eqref{eq_VF_Y+} at singular points are $-2B$ and $A_t=\frac{dA}{dt}$. Thus, $A$ and $A_t$ have simultaneous zeros if, and only if, $b_{01}\left(b_{01}+1\right)\left(b_{20}^2+b_{01}+1\right)$ $=0$ or
\begin{equation*}
8b_{20}^2\cos^2(t)+\cos^4(t)+4b_{01}\cos^2(t)+4b_{01}^2+4\cos^2(t)+8b_{01}+4=0,
\end{equation*}
which induce a natural stratification in the bifurcation plane $b_{01}b_{20}$. In particular, the restriction $-\frac{1}{2}<b_{01}<-\frac{2}{7}$ guarantees that the pairs $(b_{01},\bar{b}_{20}(b_{01}))$ are only in the strata $(-1,0)\times\mathbb{R}$ in the bifurcation plane. The singular points, under the hypothesis stated, are two nodes at $t=\frac{\pi}{2},\frac{3\pi}{2}$ and four hyperbolic saddles at the remaining points. The phase portrait 
is as shown in Fig. \ref{fig11}, (left). 
\begin{figure}[h!!!]
\centering
\includegraphics[width=8cm,clip]{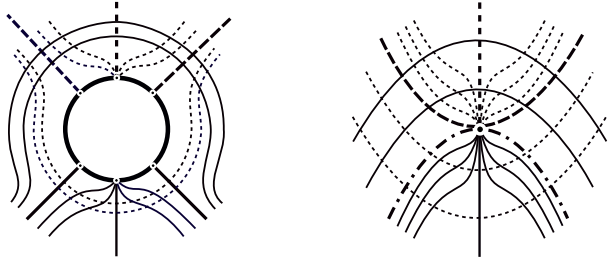}\\
\caption{ \label{fig11} (Left) Local phase portraits of $Y_i$, $(i=1,2)$ when the discriminant $\delta$ has an $A_3^+$ singularity at the Gauss cusp point. (Right) Configurations of the affine curvature lines  via blowing-down of the solutions of equation \eqref{bde_r-t}.  }
\end{figure}
The blowing-down $\varphi^*$  of the integral curves of the BDE are  as  shown in Fig. \ref{fig11}, right.

\item[2.] When $kq_{40}-4q_{21}^2<0$, the discriminant function of the BDE given by equation \eqref{BDE_N_F}has a $A_3^{-}$ singularity. The discriminant set $\delta=0$ (double $\xi$-direction set), is formed by two tangent curves given by
\begin{equation*}
v=\beta_iu^2+O(3),\quad i=1,2,\quad\text{where}\quad\beta_i=-\frac{b_{20}+(-1)^i}{b_{01}}.
\end{equation*}
The BDE in the variables $r$, $t$, after dividing by $r^3$ is given by
\begin{equation}\label{bde_r-t2}
A(t)dr^2 + 2rB(t)drdt + r^2C(t)dt^2=0,
\end{equation}
where 
\begin{eqnarray*}
A(t) & = & \cos(t)\left(\cos^4(t)+4\cos^2(t)\sin(t)b_{20}+4(b_{01}+1)\sin^2(t)\right), \\
B(t) & = & (3\cos^2(t)-2)(b_{01}\sin(t)+b_{20}\cos^2(t))-\cos^2(t)\sin(t)(\cos^2(t)-2), \\
C(t) & = & -\cos(t)(\cos^2(t)(\cos^2(t)-2)+2\cos^2(t)\sin(t)b_{20}+2b_{01}\sin^2(t)).
\end{eqnarray*}
The BDE \eqref{bde_r-t2} has six hyperbolic singularities in the interval $[0,2\pi]$, which are given by the roots of 
$$-\sin^2(t)P_i^+-\sin(t)+P_i^+=0,$$
%
%
where $P_i^+$ is a solution of the equation \eqref{eq_Par_Inv_A3+}. Consider as before the vector field defined by equation \eqref{eq_VF_Y+}. We consider the expression of $b_{20}$ in term of $b_{01}$, which for this case is given by
\begin{equation}\label{eq_b20+}
b_{20}=-\frac{\sqrt{2}}{4}\frac{4b_{01}^2+13b_{01}+4}{\sqrt{\left(2b_{01}+1\right)\left(7b_{01}+2\right)}}.
\end{equation}
Thus, $b_{01}\in\left(-\infty,-\frac{1}{2}\right)\cup\left(-\frac{2}{7},\infty\right)$. The eigenvalues of the linearization of the vector field \eqref{eq_VF_Y+} are given by $A_t$ and $-2B$. The real functions $A$ and $A_t$ have simultaneous zeros if, and only if, $b_{01}(b_{01}+1)(-b_{20}^2+b_{01}+1)=0$, or 
\begin{equation*}
8b_{20}^2\cos^2(t)+\cos^4(t)-4b_{01}\cos^2(t)+4b_{01}^2-4\cos^2(t)+8b_{01}+4=0.
\end{equation*}
%
%
The product of the eigenvalues at singular points are given by the next expressions: for $P_1^+$
\begin{equation}\label{eq_prod_e-v_p1}
\frac{3}{2}\frac{(4b_{01}-1)(2\cos^2(t)b_{01}+\cos^2(t)+7b_{01}+2)^2b_{01}^2}{(2b_{01}+1)(7b_{01}+2)^3},
\end{equation}
and for $P_2^+$
{\footnotesize
\begin{equation}\label{eq_prod_e-v_p2}
-\frac{1}{32}\frac{(4b_{01}-1)(4b_{01}^2+10b_{01}+1)(7\cos^2(t)b_{01}+8b_{01}^3+2\cos^2(t)+20b_{01}^2+16b_{01}+4)^2b_{01}^2}{(7b_{01}+2)(b_{01}+1)^5(2b_{01}+1)^3}
\end{equation}
}
The equations \eqref{eq_prod_e-v_p1} and \eqref{eq_prod_e-v_p2} give the exceptional values of $b_{01}$ where the topological type of the singularities changes. The singularities of the vector field $Y_i$ are hyperbolic saddles and nodes (it can be verified through the equations \eqref{eq_prod_e-v_p1} and \eqref{eq_prod_e-v_p2}) and the phase portraits corresponding to each case stated in the Theorem are as shown in Fig. \ref{fig15}.
\begin{figure}[H]
\centering
	\includegraphics[scale=0.5]{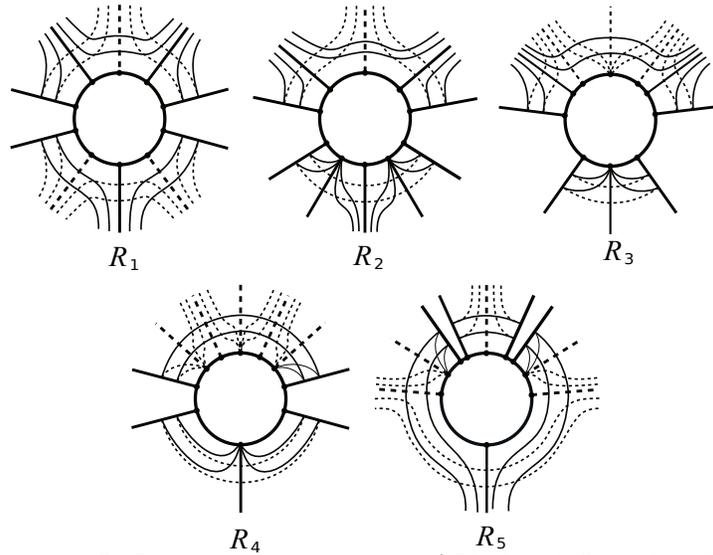}
\caption{Local phase portraits of $Y_i$, $i=1,2$, and global resolution of the affine curvature lines near a Gauss cusp point when the discriminant $\delta$ has an $A_3^-$ singularity at this point.}\label{fig15}
\end{figure}

As in the first case, the blow-up of the   affine principal configurations, given by equation \eqref{bde_r-t2}, has also six hyperbolic singularities (saddles and  nodes) and the number, types and relative position of saddle and nodes are as shown in Figs    \ref{fig8} and \ref{fig15}.

The construction of the topological equivalence between the binary differential equation of affine curvature lines and the normal forms stated can be performed using the method of canonical regions, see \cite{GS-1982}. \qed  
%
%
\end{proof}

\section*{Acknowledgements} 
 The first author is supported by the CAPES project grant number\\$88887$.$136371$$/2017$-$00$-$465591/2014$-$0$-\textit{INCT de Matem\'atica}, during a post-doctoral period at IME-UFG, Goi\^ania, Brazil. The second and third authors are fellows of CNPq.

%
%

\vskip 2cm




\section*{Appendix}

In this section the proof of Proposition \ref{aff_curv-lin_cusp_G} will be given.\\

A direct calculation shows that the  BDE of affine curvature lines in the parametrization $X(u,v)=(u,v,h(u,v))$, see equation  \eqref{eq-param-parab}, is given by
\begin{eqnarray*}
	&  & \left(a_{11}uv+a_{02}v^2+\bar{A}_3+O(4)\right)du^2+2\left(\hat{b}_{01}v+\bar{B}_2+\bar{B}_3+O(4)\right)dudv+\\
	&  & +\left(c_{10}u+c_{01}v+\bar{C}_2+\bar{C}_3+O(4)\right)dv^2=0,
\end{eqnarray*}
where the $\bar{A}_i$, $\bar{B}_j$ and $\bar{C}_k$ are homogeneous polynomials in the $u,v$ variables of degree $i,j,k$ respectively and the coefficients $a_{i_1j_1}$, $b_{i_2j_2}$ and $c_{i_3j_3}$ are functions of the  coefficients $q_{ij}$ of the parametrization $X$. We make a smooth change of coordinates of the form 
\begin{equation*}
\begin{cases*}
& $u= \alpha_1U+\alpha_2V+p_2(U,V))$, \\
& $v= \beta_1U+\beta_2V+q_2(U,V)$,
\end{cases*}
\end{equation*}
where $p_2$ and $q_2$ are homogeneous polynomials in the $U,V$ variables of degree $2$. We multiply the new BDE by $1+\gamma_1U+\gamma_2V$ and by a suitable choosing of $\alpha_i$, $\beta_1$, $\gamma_1$, $\gamma_2$, $p_2$ and $q_2$, we rewrite the BDE as
\begin{eqnarray*}
	&  & \left(\hat{A}_3(U,V)+O(4)\right)dU^2+2\left(\tilde{b}_{01}V+\tilde{b}_{20}U^2+\hat{B}_3+O(4)\right)dUdV+\\
	&  & +\left(U+\hat{C}_3+O(4)\right)dV^2=0,
\end{eqnarray*}
where $\hat{A}_3$, $\hat{B}_3$ and $\hat{C}_3$ are homogeneous polynomials of degree $3$ (for more details see \cite{Tari2007}). Finally we consider the change of coordinates
\begin{equation*}
\begin{cases*}
& $U= u+p_3(u,v))$, \\
& $V= v+q_3(u,v)$,
\end{cases*}
\end{equation*}
and again multiply the new equation by $1+\gamma_{20}u^2+\gamma_{11}uv+\gamma_{02}v^2$. 
%
By appropriate values $\gamma_{20}$, $\gamma_{11}$, $\gamma_{02}$, and the polynomials $p_3(u,v)$ and $q_3(u,v)$, obtained solving the linear homological equations, we can rewrite the BDE as
%
%
\begin{equation*}
\left(A_{30}u^3+O(4)\right)du^2+2\left(b_{01}v+\hat{b}_{20}u^2+\hat{b}_{30}u^3+O(4)\right)dudv+\left(u+O(4)\right)dv^2,
\end{equation*}
where   $b_{01}=-\frac{1}{14}\frac{9q_{21}^2-4q_{40}k}{3q_{21}^2-q_{40}k}$,
\begin{equation}\label{eq_beta20}
b_{20}=-\frac{1}{10976}\frac{411q_{21}^4-195q_{21}^2kq_{40}+20k^2q_{40}^2}{\beta_{2}^5k^5q_{21}^3(3q_{21}^2-kq_{40})^2(9q_{21}^2-4kq_{40})},
\end{equation}
\begin{equation*}
A_{30}=\frac{1}{21952}\frac{4q_{21}^2-kq_{40}}{\beta_{2}^{10}k^{10}q_{21}^4(3q_{21}^2-kq_{40})^2(9q_{21}^2-4kq_{40})^2}.
\end{equation*}
Thus, fixing the sign of $4q_{21}^2-kq_{40}$ we can reduce $A_{30}$ to $\pm1$ by a suitable choose of $\beta_2$ as we show next: if $4q_{21}^2-kq_{40}<0$, for
\begin{equation*}
\beta_2=\beta_2^-=\left(\frac{1}{21952}\frac{-\left(4q_{21}^2-kq_{40}\right)}{A_{30}k^{10}q_{21}^4(3q_{21}^2-kq_{40})^2(9q_{21}^2-4kq_{40})^2}\right)^{\frac{1}{10}}
\end{equation*}
we obtain $A_{30}=-1$. In the case $4q_{21}^2-kq_{40}>0$, for
\begin{equation*}
\beta_2=\beta_2^+=\left(\frac{1}{21952}\frac{4q_{21}^2-kq_{40}}{A_{30}k^{10}q_{21}^4(3q_{21}^2-kq_{40})^2(9q_{21}^2-4kq_{40})^2}\right)^{\frac{1}{10}},
\end{equation*}
we have $A_{30}=1$. Replacing $\beta_2^-$ and $\beta_2^+$ in the equation \eqref{eq_beta20}, we obtain the desired expresions and this complete the proof. \qed
\begin{remark}
In the proof of Theorem \ref{Thm_confg_princ_affins} we use other expressions for $b_{20}$ and $\bar{b}_{20}$ in each case. Note that $b_{01}$, $\bar{b}_{20}$ and $b_{20}$ are expressed in terms of $q_{21}$ and $q_{40}$. In fact, we replace the values of $b_{01}$ in $b_{20}$ and $\bar{b}_{20}$, and obtain the next expressions to them, in terms only of $b_{01}$:
{\small
\begin{equation*}
\bar{b}_{20}=-\frac{\sqrt{2}}{4}\frac{(4b_{01}^2+13b_{01}+4)}{\sqrt{-(2b_{01}+1)(7b_{01}+2)}}\quad\text{and}\quad b_{20}=-\frac{\sqrt{2}}{4}\frac{(4b_{01}^2+13b_{01}+4)}{\sqrt{(2b_{01}+1)(7b_{01}+2)}}.
\end{equation*}
}
We have $\bar{b}_{20}$ when $4q_{21}^2-kq_{40}<0$ and $b_{20}$ when $4q_{21}^2-kq_{40}>0$, respectively singularities of the discriminant function of associate binary differential equation of types $A_3^+$ or $A_3^-$ at the origin.
\end{remark}

\end{document}